\documentclass[12pt]{amsart}
\usepackage{multicol,mathtools}

\usepackage{youngtab}

\usepackage[colorlinks=true,linkcolor=blue,urlcolor=blue,citecolor=blue, pagebackref]{hyperref}
\usepackage[alphabetic]{amsrefs}
\usepackage{soul}

\usepackage[margin=1in]{geometry}

\usepackage{times}

\usepackage[toc,page]{appendix}
\usepackage{epsfig,color,wrapfig,epstopdf}
\usepackage{verbatim,amscd,color}
\usepackage{tikz}
\usepackage{tikz-cd}
\usetikzlibrary{matrix}
\usepackage{youngtab, ytableau}
\usepackage{graphicx}
\usepackage{array}
\usepackage{indentfirst,pict2e}
\usepackage[leqno]{amsmath}
\usepackage{latexsym}
\usepackage{youngtab, ytableau}
\usepackage{algorithm}
\usepackage{algpseudocode}
\usepackage{pifont}
\usepackage{tikz}

\usepackage{multido}
\usepackage{amsmath,amsthm,amsfonts,amssymb}
\usepackage[mathscr]{euscript}
\usepackage{enumerate}
\newcommand\Fl{\operatorname{Fl}}

\newcommand{\xddots}{%
  \raise 4pt \hbox {.}
  \mkern 6mu
  \raise 1pt \hbox {.}
  \mkern 6mu
  \raise -2pt \hbox {.}
}

\definecolor{todo}{rgb}{.80,.20,.20}
\definecolor{e-mail}{rgb}{0,.40,.80}
\definecolor{reference}{rgb}{.10,.40,.42}
\definecolor{mrnumber}{rgb}{.80,.40,0}
\definecolor{citation}{rgb}{0,.40,.80}
\usepackage{fancyhdr}
\usepackage[all]{xy}
\usepackage{graphicx}

\newcommand{\ov}[1]{\overline{#1}}
\newcommand{\op}[1]{\operatorname{#1}}
\newcommand{\ovop}[1]{\ov{\op{#1}}}

\newcommand{\C}{\mathbb{C}} 

\newcommand{\Z}{\mathbb{Z}}

\newcommand{\Gr}{\operatorname{Gr}}
\newcommand{\QH}{\operatorname{QH}}
\renewcommand{\H}{\operatorname{H}}
\newcommand{\codim}{\operatorname{codim}}

\newcommand{\M}{\overline{\operatorname{M}}}

\newcommand{\sL}{\mathfrak{sl}}
\newtheorem{theorem}{Theorem}[section]

\newtheorem{corollary*}{Corollary}

\newtheorem{acknowledgement*}[theorem]{Acknowledgement}

\newtheorem{project*}{Project}
\newtheorem{projectone*}{Project One}
\newtheorem{projecttwo*}{Project Two}
\newtheorem{projectthree*}{Project Three}

\newtheorem{lemma}[theorem]{Lemma}

\newtheorem{Def/Prop}{Definition/Proposition}
\newtheorem{remark/definition}{Remark/Definition}
\newtheorem{example}[theorem]{Example}
\newtheorem{proposition}[theorem]{Proposition}

\newtheorem{defn/thm}[theorem]{Definition/Theorem}
\newtheorem{definition/lemma}[theorem]{Definition/Lemma}
\newtheorem{definition}[theorem]{Definition}

\newtheorem{thm}{Theorem}

\newtheorem*{conjecture*}{The GW $\equiv$ CB Conjecture}

\theoremstyle{remark}
\newtheorem{remark}[theorem]{Remark}

\usepackage{soul}

\definecolor{amethyst}{rgb}{0.6, 0.4, 0.8}
\definecolor{kellygreen}{rgb}{0.3, 0.73, 0.09}

\makeatletter
\@namedef{subjclassname@2010}{%
\textup{MSC2010}}
\makeatother

\begin{document}

\title[The GW $\equiv$ CB conjecture]{On an equivalence of divisors on $\M_{0,n}$ from Gromov-Witten theory and Conformal Blocks} 

\author[L.~Chen]{L.~Chen}
\address{Linda Chen \newline \indent Department of Mathematics and Statistics, Swarthmore College, Swarthmore, PA 19081}
\email{lchen@swarthmore.edu}

\author[A.~Gibney]{A.~Gibney}
\address{Angela Gibney\newline \indent Department of Mathematics, Rutgers University, Piscataway, NJ 08854\newline \indent Department of Mathematics, University of Pennsylvania, Philadelphia, PA 19104-6395}
\email{angela.gibney@gmail.com}

\author[L.~Heller]{L.~Heller}
\address{Lauren Cranton Heller \newline \indent  Department of Mathematics, University of California, Berkeley, CA 94720}
\email{lch@math.berkeley.edu}

\author[E.~Kalashnikov]{E.~Kalashnikov}
\address{Elana Kalashnikov \newline \indent  Department of Mathematics, Harvard University, Cambridge, MA 02138 \newline \indent Faculty of Mathematics, Higher School of Economics, Moscow, Russia, 101000}
\email{kalashnikov@math.harvard.edu}

\author[H.~Larson]{H.~Larson}
\address{Hannah Larson \newline \indent  Department of Mathematics, Stanford University, Stanford, CA 94305}
\email{hlarson@stanford.edu}

\author[W.~Xu]{W.~Xu}
\address{Weihong Xu \newline \indent  Department of Mathematics, Rutgers University, Piscataway, NJ 08854}
\email{wx61@math.rutgers.edu}

\thanks{LC was partially supported by Simons Collaboration Grant 524354 and NSF Grant DMS-2101861. AG was partially supported by NSF Grant DMS-1902237. HL was partially supported by the Hertz Foundation and NSF GRFP Grant DGE-1656518.}


\begin{abstract}We consider a conjecture that identifies two types of base point free divisors on $\M_{0,n}$. The first arises from  Gromov-Witten theory of a Grassmannian.  The second comes from first Chern classes of vector bundles associated to simple Lie
algebras in type A. Here we
reduce  this conjecture  on $\M_{0,n}$ to the same statement for $n=4$.  A reinterpretation  leads to a proof of the conjecture on $\M_{0,n}$ for a large class, and we give sufficient conditions for the non-vanishing of these divisors.
 \end{abstract}
\date{}
\maketitle

\section{Introduction}
The moduli space $\M_{g,n}$   of $n$-pointed stable curves of genus $g$ is a fundamental object that gives insight into smooth curves and their degenerations. A projective variety such as $\M_{g,n}$  can be better understood by investigating its base point free divisors, which give rise to morphisms.   Moduli spaces of curves for different $g$ and $n$ are connected through tautological clutching and projection morphisms which impart a rich combinatorial structure. Cycles on $\M_{g,n}$  reflect this, and often are governed by recursions, and amenable to inductive arguments. Consequently, many questions can be reduced to  moduli of curves of smaller genus and fewer marked points. 

We study two families of base point free divisors on the smooth projective variety $\M_{0,n}$. The first are obtained from the Gromov-Witten theory of Grassmannians, and the second are first Chern classes of  globally generated vector bundles defined by  representations of a simple Lie algebra in type A, so-called conformal blocks divisors.  While quite different, in some cases they are given by the same data and believed to be numerically equivalent (see the {\em{GW $\equiv$ CB Conjecture}}).  The identification of   characteristic classes of  vector bundles with classes of geometric loci is interesting as it can lead to valuable information about associated maps and cones of divisors.

We prove two main results. In Theorem~\ref{thm:ConjReduction},  we show the GW $\equiv$ CB Conjecture on $\M_{0,n}$ can be reduced to the case $n=4$ by using the fact that both types of classes satisfy a factorization property with respect to pullback along tautological maps. In Theorem~\ref{GWCB}, we show the GW $\equiv$ CB Conjecture for divisors satisfying what we call the  \emph{column condition} (see Definition~\ref{ColumnDef}). As an application, in  Proposition~\ref{nonvanishing}, we give sufficient criteria for the non-vanishing of the GW and CB divisors, and in particular, conditions that guarantee their associated maps are nonconstant.

We next state the GW $\equiv$ CB Conjecture, and our results in more detail.  We also describe our methods and approach, which are varied, drawing from a  variety of techniques and facts
from Gromov-Witten theory and the theory of conformal blocks.

Given a collection of partitions $\lambda^\bullet=(\lambda^1,\dots,\lambda^n)$
satisfying $\sum_i|\lambda^i|=(r+1)(l+1)$, we obtain a \emph{GW-divisor}  $I^{1,\Gr_{r,r+l}}_{1,\lambda^{\bullet}}$ on $\M_{0,n}$ (see \S\ref{sec:GWclasses}). The same data 
determines $n$ simple modules over the Lie algebra $\sL_{r+1}$ and defines a vector bundle of coinvariants $\mathbb{V}(\sL_{r+1}, \lambda^{\bullet}, l)$ on $\M_{g,n}$ \cite{tuy},  which is globally generated on $\M_{0,n}$ \cite{fakhr}. The condition $\sum_i|\lambda^i|=(r+1)(l+1)$ means $\mathbb{V}(\sL_{r+1}, \lambda^{\bullet}, l)$ is {\em{critical level}} (see \S\ref{CL}). 

\smallskip
Such GW-divisors and critical level CB-bundles
are believed to be related:

\begin{conjecture*}\cite[Question~3.3]{BG}\label{ConG} Let  $\lambda^{\bullet}=(\lambda^1, \ldots, \lambda^n)$ be  partitions corresponding to Schubert classes in $\Gr_{r,r+l}$ such that   $\sum_i|\lambda^i|=(r+1)(l+1)$.  Then the GW-divisor  $I^{1,\Gr_{r,r+l}}_{1, \lambda^{\bullet}}$ on $\M_{0,n}$
 is
 numerically equivalent to the first Chern class of the critical level CB-bundle $\mathbb{V}(\sL_{r+1}, \lambda^{\bullet}, l)$.
\end{conjecture*}

The GW $\equiv$ CB Conjecture was proved for the case $l=1$ in \cite[Theorem~3.1]{BG}.  Note that Remark 3.2 and Question 3.3 of \cite{BG}  referred to the Grassmannians $\Gr_{1,r+1}$ and $\Gr_{l,r+l}$, respectively, but   instead correspond to the Grassmannians $\Gr_{r,r+1}$ and $\Gr_{r,r+l}$ in our notation.

\smallskip
Our first main result is to reduce the GW $\equiv$ CB Conjecture to the $n=4$ case. 
\begin{thm}\label{thm:ConjReduction}
GW $\equiv$ CB   on $\M_{0,4}$ implies that GW $\equiv$ CB on $\M_{0,n}$, for all $n\ge 4$.
\end{thm}

On $\M_{0,4}\cong \mathbb{P}^1$, the first Chern class is  the degree of the bundle.  We verify the GW $\equiv$ CB Conjecture for a class of divisors defined by partitions satisfying the following:
\begin{definition}\label{ColumnDef}Let $\#\lambda$ be the number of nonzero rows of a partition $\lambda$ or, equivalently, the height of the first column, so $\#\lambda= \lambda^T_1$ where $\lambda^T$ is the transpose to $\lambda$. We say that $\lambda^{\bullet}$ satisfies {\em{the column condition}} if  $\sum_{i=1}^n|\lambda^i|=(r+1)(l+1)$, and $\sum_{i=1}^n \#\lambda^i \le 2(r+1)$.\end{definition}
\begin{thm}\label{GWCB}
GW $\equiv$ CB holds on $\M_{0,n}$ if $\lambda^\bullet$ satisfies the column condition.
\end{thm}
We reduce Theorem~\ref{GWCB} to the $n=4$ case in Proposition~\ref{prop:ReductionColumnId}, and then establish the $n=4$ case in Proposition~\ref{pro:conjspecialcase}.  
If $\sum_{i=1}^n \#\lambda^i < 2(r+1)$, both the GW and CB classes are trivial.  In \S\ref{KnownCases} we give an infinite family of nontrivial examples satisfying Theorem~\ref{GWCB}. 
In addition, with ConfBlocks, a package for Macaulay2,
we check the GW $\equiv$ CB Conjecture holds for small values of $r$ and $l$ by verifying it on $\M_{0,4}$ (Proposition~\ref{M2}).

Both critical level CB-bundles and GW-divisors satisfy symmetries: By \cite[Prop. 1.6]{BGMA}, $c_1(\mathbb{V}(\sL_{r+1}, \lambda^{\bullet}, l))\equiv c_1(\mathbb{V}(\sL_{l+1}, (\lambda^T)^{\bullet}, r))$; similarly, $I^{1,\Gr_{r,r+l}}_{1,\lambda^{\bullet}}\equiv I^{1,\Gr_{l,r+l}}_{1,(\lambda^T)^{\bullet}}$, from isomorphisms $\Gr_{r,r+l}\cong \Gr_{l,r+l}$.  Thus, for triples $(\lambda^{\bullet}, r,l)$ for which the  GW $\equiv$ CB Conjecture holds,
\begin{equation}\label{BigID}I^{1, \Gr_{r, r+l}}_{1, \lambda^\bullet} \equiv c_1(\mathbb{V}(\sL_{r+1},\lambda^{\bullet},l)) \equiv  I^{1, \Gr_{l, r+l}}_{1, (\lambda^T)^\bullet} \equiv c_1(\mathbb{V}(\sL_{l+1},(\lambda^T)^{\bullet},r)).\end{equation}
Therefore, Theorem~\ref{GWCB} also proves the conjecture when an analogous row condition is satisfied.

 To show that $I^{1,\Gr_{r,r+l}}_{1, \lambda^{\bullet}}$ and $c_1(\mathbb{V}(\sL_{r+1},\lambda^{\bullet},l))$ are numerically equivalent, it suffices to show they intersect all $F$-curves, which span $H_{2}(\ovop{M}_{0,n})$, in the same degree.   Formulas for these intersections have the same shape (see \eqref{GWE} and \eqref{CLEP}).  A comparison of their constituent parts gives the reduction of the GW $\equiv$ CB Conjecture to $\M_{0,4}$.
 This comparison relies on Witten's Dictionary (\S\ref{WD}), which gives the rank of a CB-bundle in terms of a computation in the cohomology ring of a Grassmannian.
 
 Using Proposition~\ref{prop:lines}, we provide an alternative  characterization of the GW $\equiv$ CB Conjecture for $n=4$, reinterpreting such classes as intersection numbers on two-step flag varieties.  We show using Proposition~\ref{Geometric} that for partitions satisfying the column condition, the GW-class on $\M_{0,4}$ can be identified with an intersection of Schubert classes on a two-step flag variety, and with this, prove Proposition~\ref{pro:identity}, the key identity on the GW-side of the story.  
    
 Proposition~\ref{V:identity} is the identity on the other side of the story, giving a relation for first Chern classes of critical level CB-bundles analogous to Proposition~\ref{pro:identity}.  The proof depends on rank conditions, which we check with Witten's Dictionary, quantum cohomology, and Schubert calculus. 

As we show in Proposition~\ref{nonvanishing}, our proof of Theorem~\ref{GWCB} gives sufficient criteria for the non-vanishing of GW and CB divisors. Proposition~\ref{nonvanishing} partially answers the question of 
finding necessary and sufficient conditions for non-vanishing of CB divisors, asked in \cite{BGMB}. If such globally generated divisors were numerically equivalent to zero, then their associated maps would be constant.  In particular, establishing that the  divisors are nonzero is the first step to finding  potentially interesting morphisms.

One reason for interest in identifications of classes (such as the GW-classes) that arise as geometric loci, with characteristic classes of globally generated vector bundles (such as the critical level CB-bundles) is that we can hope to gain some information about the morphisms they determine. We know in case $l=1$ or $r=1$ that such morphisms have images with modular interpretations as (weighted) points supported on Veronese curves \cite{Giansiracusa}, 
\cite{GiansiracusaSimpson}, \cite{gg}, \cite{gjms}. Moreover, identities like that predicted by the GW $\equiv$ CB Conjecture constrain the number of potentially independent 
extremal rays of the cone of nef divisors, giving evidence that it may be polyhedral, as predicted \cite{gkm}, in spite of the large numbers of nef divisors given by GW-divisors and first Chern classes of vector bundles of coinvariants. Both constructions give rise to basepoint free cycles of arbitrary codimension, and in  \cite{ICERM2}, which is ongoing, we are
considering the problem of their extremality  in cones of positive cycles.

The GW-classes we work with here are an example of a more general class of basepoint free Gromov-Witten loci $I^{c,X}_{d,\alpha^\bullet}$ of codimension c in $\ovop{M}_{0,n}$, defined in \cite{BG} from a homogeneous variety $X=G/P$ and a collection of Schubert subvarieties of $X$ satisfying particular numerical conditions. 

We study Chern classes of vector bundles that are special
cases of sheaves  $\mathbb{V}(\mathfrak{g},\{\mathcal{W}^i\},l)$, constructed  from simple modules $\mathcal{W}^i$ over a simple Lie algebra $\mathfrak{g}$. Fibers are vector spaces of coinvariants, and their duals are vector spaces  of conformal blocks.  The bundles satisfy factorization,  a property originally detected by Tsuchiya and Kanie  \cite{TK} in the case conformal blocks were defined on $\mathbb{P}^1$ by  $\sL_2$-modules.  Tsuchiya, Ueno, and Yamada constructed the sheaves on a space parametrizing stable pointed curves with coordinates, showing they satisfy factorization, and are vector bundles \cite{tuy}.  Tsuchimoto in \cite{ts} proved they are coordinate free and descend to $\ovop{M}_{g,n}$.  These are referred to in the literature as Verlinde bundles, vector bundles of coinvariants, vector bundles of covacua, and vector bundles of conformal blocks.   A notable feature is that (duals of) their fibers, vector spaces of conformal blocks, are canonically isomorphic to generalized theta functions \cites{bl1, KNR, Faltings, LaszloSorger}.  Fakhruddin, in \cite{fakhr}, extended an argument of \cite{tuy} for smooth, pointed curves of genus zero with coordinates, to show they are globally generated on $\ovop{M}_{0,n}$.  Their Chern classes have subsequently been studied, including in \cites{ gg, Giansiracusa, gjms,  ags, MOP, BGMA, BGMB, BGK, BG, MOPPZ}.

\section{Background and Notation}

\subsection{Schubert calculus}\label{sec:LR}
For positive integers $r$ and $l$, let $\Gr_{r,r+l}$ denote the Grassmannian of \(r\)-planes in \(\C^{r+l}\). This is a smooth projective homogeneous variety of dimension $rl$. Schubert varieties \(X_\lambda\) are certain special subvarieties of $\Gr_{r,r+l}$  indexed by partitions in the $r\times l$ rectangle $(l^r) = (l, \ldots, l)$. Each such partition is a weakly decreasing sequence of at most \(r\) integers between \(0\) and \(l\), and we identify partitions that differ by a number of trailing \(0\)'s. A partition can be represented as a Young diagram with $\lambda_i$ boxes in the $i$th row, where the rows are labelled from top to bottom. We use sequence notations and Young diagrams interchangeably. \(X_\lambda\) has codimension \(|\lambda|\coloneqq\sum \lambda_i\). Each Schubert variety \(X_\lambda\) determines a cohomology class \(\sigma_\lambda\in\H^{2|\lambda|}\Gr_{r,r+l}\). These classes form a \(\Z\)-basis for the cohomology ring $\H^*\Gr_{r,r+l}$. The complement of the Young diagram of \(\lambda\), read from bottom to top, gives the {\it dual partition} \(\lambda^\vee\).

Schur polynomials \(\{s_\lambda\}\) form a \(\Z\)-basis for the ring of symmetric functions \(\Lambda\).  We write
\begin{equation}\label{eq:LR}
 s_{\lambda^1}\cdot s_{\lambda^2}\cdots s_{\lambda^n}= \sum_\nu c_{\lambda^\bullet}^\nu s_\nu,
 \end{equation} where $c_{\lambda^\bullet}^\nu$ are the \emph{generalized Littlewood-Richardson coefficients}, and we note that  $c_{\lambda^\bullet}^\nu=0$ unless $\sum_i |\lambda^i|=|\nu|$. When $n=2$, this gives the usual
Littlewood-Richardson coefficients $c_{\lambda^1,\lambda^2}^\nu$. 

There is a surjective ring homomorphism $\Lambda\rightarrow \H^*\Gr_{r,r+l}$
defined by 
\[
s_\lambda \mapsto \left\{ \begin{array}{cl} \sigma_\lambda & \text{ if } \lambda\subseteq (l^r) \\ 0 &  \text{ if } \lambda\not\subseteq (l^r) \end{array}. \right .
\]In particular, given a collection of partitions  $\lambda^\bullet=(\lambda^1,\dots,\lambda^n)$, each  contained in an $r\times l$ rectangle $(l^r)$, the product of Schubert classes $\sigma_{\lambda^i}\in \H^{2|\lambda^i|}\Gr_{r,r+l}$  is given by 
\[\sigma_{\lambda^1}\cdot\sigma_{\lambda^2}\cdots\sigma_{\lambda^n} = \sum_\nu c_{\lambda^\bullet}^\nu \sigma_\nu,\]
where we sum over $\nu$ such that $\sum_i |\lambda^i|=|\nu|$ and $\nu\subseteq (l^r)$, and $c_{\lambda^\bullet}^\nu$ are the generalized Littlewood-Richardson coefficients in \eqref{eq:LR}. Observe also that for $\nu\subseteq (l^r)$, we have $c_{\lambda^\bullet}^\nu =\int_{\Gr_{r, r+l}}\prod_{i=1}^{n}\sigma_{\lambda^i}\cdot\sigma_{\nu^\vee}$.

In \S\ref{a:LR}, we state and prove some  facts about Littlewood-Richardson coefficients that we will use in the proofs of our main results. For example, we show in Lemma~\ref{lem:LR} a useful factorization identity that is a special case of such identities for Littlewood-Richardson coefficients on the bounday of the cone given by Horn inequalities.

\subsection{GW-classes and GW-invariants on \texorpdfstring{$\M_{0,n}$}{M 0,n bar}}\label{sec:GWclasses}

Let $\M_{0,n}(\Gr_{r,r+l},d)$ denote the Kontsevich moduli space of genus zero degree $d$ stable maps to $\Gr_{r,r+l}$.   This parametrizes data $(f,C,p_1,\dots,p_n)$, where $C$ is a connected nodal curve of genus $0$, and $f:C\to \Gr_{r,r+l}$ is a map such that $f_*[C]=d$ in $\H_2\Gr_{r,r+l}$. This space of stable maps is an irreducible projective variety of dimension $n-3+(r+l)d+rl$ that comes with $n$ evaluation maps 
$ev_i: \M_{0,n}(\Gr_{r,r+l},d) \longrightarrow \Gr_{r,r+l},$ given by sending $(f,C,p_1,\dots,p_n)$ to $f(p_i)$. 
Given a collection of partitions $\lambda^\bullet=(\lambda^1,\dots,\lambda^n)$, each contained in an $r\times l$ rectangle, consider the collection of Schubert varieties $X_{\lambda^i}$ and Schubert classes $\sigma_{\lambda^i}\in \H^{2|\lambda^i|}(X)$. Under the (flat) map $\eta:\M_{0,n}(\Gr_{r,r+l},d)\to  \M_{0,n}$ that sends $(f,C,p_1,\dots,p_n)$ to $(C,p_1,\dots,p_n)$,
since  $\dim \M_{0,n} =n-3$, we have
\[c\coloneqq\codim \eta_*(\cap_{i\in [n]}ev_{i}^*\sigma_{\lambda^i})=\sum_{i\in [n]}|\lambda^i|-(r+l)d-rl.\]

We define the \emph{GW-class} of  codimension $c$ on  $\M_{0,n}$ as
 \begin{equation}\label{GWLocus} 
I^{c,X}_{d,\lambda^{\bullet}}\coloneqq\eta_*(\cap_{i\in [n]}ev_{i}^*\sigma_{\lambda^i}).
  \end{equation}
This is a base point free cycle on $\M_{0,n}$ \cite{BG}. These classes are called \emph{GW-divisors} when they are of codimension $c=1$. In particular, when  $d=1$ and the collection $\lambda^\bullet$ satisfies:  \begin{equation}\label{CLequation}\sum_{i\in [n]}|\lambda^i|=(r+l) + rl+1=(r+1)(l+1),\end{equation}
we obtain GW-divisors   $I^{1,\Gr_{r,r+l}}_{1,\lambda^{\bullet}}$  on $\M_{0,n}$. The condition in \eqref{CLequation} is called the  \emph{critical level condition}.

Using the identification of the bottom and top cohomology groups with $\mathbb{Z}$, when $d=0$ and $c=0$, we obtain generalized  Littlewood-Richardson coefficients of \S\ref{sec:LR}:
\begin{equation}\label{eq:classical}
I^{0,\Gr_{r,r+l}}_{0,\lambda^{\bullet}}  = c_{\lambda^\bullet}^{(l^r)} =\int_{\Gr_{r,r+l}} \prod \sigma_{\lambda^i}.
\end{equation}
Similarly when $\sum_{i=1}^n|\lambda^i|=(r+l)d + rl+n-3$, the GW-classes $I^{n-3,\Gr_{r,r+l}}_{1,\lambda^\bullet}$ of codimension $n-3$ on $\M_{0,n}$ are  the \emph{$n$-pointed Gromov-Witten invariants} 
\begin{equation}\label{eq:GWinvariant}I_d(\sigma_{\lambda^1},\dots,\sigma_{\lambda^n}) = I^{n-3,\Gr_{r,r+l}}_{d, \lambda^{\bullet}}.
\end{equation} 

\subsection{Quantum cohomology of the Grassmannian}\label{sec:QH}
The (small) quantum cohomology ring of the Grassmannian $\Gr_{r,r+l}$ is defined as module over $\mathbb{Z}[q]$ by  $\QH^*\Gr_{r,r+l}\coloneqq \H^*\Gr_{r,r+l}\otimes_{\mathbb{Z}} \mathbb{Z}[q]$.  There is a $\mathbb{Z}[q]$-basis of Schubert classes $\sigma_\lambda\otimes 1$, which we also denote by $\sigma_\lambda$ in an abuse of notation.  There is a quantum product that defines an associative ring structure on the graded ring $\QH^*\Gr_{r,r+l}$, where $\sigma_\lambda$ has degree $|\lambda|$ and $q$ has degree $r+l$ \cite{bertram}. The quantum product is defined by:
\[\sigma_{\lambda^1}*\sigma_{\lambda^2} = \sum_{\nu,d} c_{\lambda^1,\lambda^2}^{d,\nu} q^d\sigma_\nu,\]
where $c_{\lambda^1,\lambda^2}^{d,\nu}$ is the 3-pointed Gromov-Witten invariant $I_d(\sigma_{\lambda^1},\sigma_{\lambda^2},\sigma_{\nu^\vee})$, where \(\nu^\vee\) is the partition dual to \(\nu\) defined in \S\ref{sec:LR}.

Since the $\sigma_\lambda$ form a basis for $\QH^*\Gr_{r,r+l}$ as a $\mathbb{Z}[q]$-module, we can write
\begin{equation}\label{eq:quantum-product} 
\sigma_{\lambda^1}*\dots*\sigma_{\lambda^n} = \sum_{\nu,d} c_{\lambda^\bullet}^{d,\nu} q^d\sigma_\nu. \end{equation}
We call these structure coefficients $c_{\lambda^\bullet}^{d,\nu}$ the  degree $d$ \emph{quantum  Littlewood-Richardson coefficients}. Note that $c_{\lambda^\bullet}^{d,\nu}=0$ unless $\sum |\lambda^i| = |\nu|+(r+l)d$. Note also that the quantum  Littlewood-Richardson coefficients $c_{\lambda^\bullet}^{d,\nu}$ are in general not Gromov-Witten invariants themselves,  though they are determined by the 3-pointed Gromov-Witten invariants.

By the Main Lemma of \cite{ciocanpaper}, quantum products can be obtained by first computing classical products and then removing rim-hooks. We state the Main Lemma here for the convenience of the reader. We first define classes $\sigma_\lambda$ for all partitions $\lambda$, not just those fitting into an $(l^r)$ rectangle: for any nonempty partition $\lambda=(\lambda_1,\dots,\lambda_s)$, let
\[ \sigma_\lambda = \det(\sigma_{\lambda_i+j-i})_{1\leq i,j\leq s} \in \QH^{2|\lambda|}\Gr_{r,r+l};\] here \(\sigma_p=0\) for \(p<0\) and \(\sigma_p=\sigma_{(p)}\) for \(p\geq 0\).
When $\lambda$ fits into an  $(l^r)$ rectangle, this gives the (quantum) Schubert class $\sigma_\lambda$ \cite{bertram}.

An $m$-rim-hook of a partition is defined to be a collection of $m$ boxes in a partition, which start at the bottom of a column and move right and up along the rim. An $m$-rim-hook is \emph{illegal} if once removed, what remains is not a partition. The width $w$ of an $m$-rim-hook is the number of columns it occupies. 
\begin{figure}[h!]
\centering
\begin{tikzpicture}[scale=0.3]
\draw (0, 0) -- (7, 0);
\draw (0, -1) -- (7, -1);
\draw (7,0) -- (7, -1);
\draw (6,0) -- (6, -1);
\draw (5,0) -- (5, -1);
\draw (0, -2) -- (4, -2);
\draw (0, 0) -- (0, -5);
\draw[fill = blue!50]  (0, -6) rectangle (2, -5); 
\draw[fill = blue!50] (1, -5) rectangle (2, -3);
\draw [fill=blue!50] (1, -2) rectangle (4, -3);
\draw (1, 0) -- (1, -2);
\draw (0, -3) -- (1, -3);
\draw (0, -4) -- (2, -4);
\draw (1, -5) -- (1, -6);
\draw (2, 0) -- (2, -3);
\draw (3, 0) -- (3, -3);
\draw (4, 0) -- (4, -2);
\end{tikzpicture}
\caption{A $7$-rim-hook of width $4$.}
\end{figure}
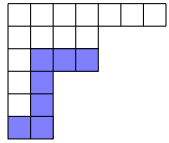

\begin{lemma}[Main Lemma, \cite{ciocanpaper}] \label{mainlemma}
Let $\lambda$ be a partition. The following is true in \(\QH^*\Gr_{k,m}\): If $\lambda$ contains an illegal $m$-rim-hook, or if $\lambda_{k+1}>0$ and $\lambda$ does not contain an $m$-rim-hook, then $\sigma_\lambda=0$.  If $\mu$ is the result of removing an $m$-rim-hook of width $w$ from $\lambda,$ then 
$\sigma_\mu=(-1)^{w+m-k} q\sigma_\lambda.$
\end{lemma}  
We use this formulation of computing classically and removing rim-hooks in the proof of Lemma~\ref{newwitten}, which is a critical ingredient for the proof of Theorem~\ref{GWCB}. 

\subsection{CB-Bundles in type A, ranks, critical level vanishing and identities}\label{WittenD}
Partitions $\lambda^\bullet$ for $\Gr_{r,r+l}$ parametrize simple $\sL_{r+1}$-modules, and collections of  partitions give rise to vector bundles $\mathbb{V}(\sL_{r+1}, \lambda^\bullet, l)$ constructed from representations of the simple Lie algebra $\sL_{r+1}$ at level $l$. These were originally defined by Tsuchiya, Ueno, and Yamada in \cite{tuy}.  For the bundle to be nontrivial, $(r+1)$ must divide the total sum $\sum_{i=1}^n|\lambda^i|$. On $\M_{0,n}$ such bundles are globally generated \cite{fakhr}.  The following result allows one to obtain their ranks via (quantum) cohomology of Grassmannians \cite{BHorn}.

 \begin{theorem}\label{WD}(Cohomological form of Witten's Dictionary)
Let $\lambda^\bullet$ be a collection of $n$ partitions contained in an $r\times l$ rectangle satisfying $\sum_{i=1}^n |\lambda^i|=(r+1)(l+s)$ for some $s\in \mathbb{Z}$. Then the rank $R$ of $\mathbb{V}(\sL_{r+1}, \lambda^\bullet, l)$ 
on $\M_{0,n}$ may be computed as follows:
\begin{enumerate}
\item If $s\le 0$, then $R$ 
is equal to  
\[\int_{\Gr_{r+1,r+1+l+s}} \sigma_{\lambda^1}\cdot\sigma_{\lambda^2}\cdot\dots\cdot\sigma_{\lambda^n} =c_{\lambda^\bullet}^{(l+s)^{r+1}},\]
where the second equality follows from \eqref{eq:classical}.
\item If $s\ge 0$, then  $R$ is equal to $c^{s,(l^{r+1})}_{\lambda^\bullet,(l)^s}$. As in \eqref{eq:GWinvariant}, this is
the coefficient of $q^s \sigma_{(l^{r+1})}$ in the quantum product
\[\sigma_{\lambda^1} * \sigma_{\lambda^2} * \cdots * \sigma_{\lambda^n} * \sigma_{(l)}^s \text{ in } \QH^*\Gr_{r+1,r+1+l}.
\]
Here, $\sigma_{(l^{r+1})}$ is equal to the point class $[pt]$. \end{enumerate}
\end{theorem}

\subsection{Critical level vanishing and identities}\label{CL}
Suppose we are given a collection of $n$ partitions
 $\lambda^{\bullet}=(\lambda^1, \ldots, \lambda^n)$ for $\sL_{r+1}$, and as is standard, we suppose that $r+1$ divides the sum $\sum_{i=1}^n|\lambda^i|$ (see Remark~\ref{RankZero}). Following
  \cite[Def~1.1]{BGMA}, we define the critical level for the pair 
  $(\sL_{r+1},\lambda^{\bullet})$ to be 
  \[c(\sL_{r+1},\lambda^{\bullet})=-1+\frac{1}{r+1}\sum_{i=1}^n|\lambda^i|.\]
 We say that the bundle 
 $\mathbb{V}(\sL_{r+1},\lambda^{\bullet},l)$ is  
 \begin{enumerate}
 \item {\em{at the critical level}} when   $\sum_i|\lambda^i|=(r+1)(l+1)$, so $l=c(\sL_{r+1},\lambda^{\bullet})$
\item  {\em{above the critical level}} when $\sum_i|\lambda^i|=(r+1)(l+s)$ for $s\le 0$, so 
$l>c(\sL_{r+1},\lambda^{\bullet})$.
\end{enumerate}
A bundle $\mathbb{V}(\sL_{r+1}, \lambda^\bullet, l)$ that is at the critical level will be referred to as a critical level bundle.

By \cite[Prop~1.6]{BGMA}, critical level bundles satisfy  identities: 
\begin{equation}\label{CSym}c_1(\mathbb{V}(\sL_{r+1}, \lambda^\bullet, l))=c_1(\mathbb{V}(\sL_{l+1}, (\lambda^T)^\bullet, r)),\end{equation}
where $(\lambda^\bullet)^T$ denotes the collection of $n$ partitions each transpose to $\lambda^i$. 

 Furthermore, when $l>c(\sL_{r+1},\lambda^{\bullet})$, then by \cite[Prop~1.3]{BGMA}, 
 $c_1(\mathbb{V}(\sL_{r+1}, \lambda^\bullet, l))=0.$

\section{Reductions to the 4-pointed case}

In this section prove two reduction results. We first prove Theorem~\ref{thm:ConjReduction}, which reduces the GW $\equiv$ CB Conjecture to the $n=4$ case. With similar ideas, we prove Proposition~\ref{prop:ReductionColumnId}, which reduces Theorem~\ref{GWCB} to the $n=4$ case.    

\subsection{Proof of Theorem~\ref{thm:ConjReduction}}
Partitions $\lambda^{\bullet}=(\lambda^1, \ldots, \lambda^n)$ satisfying $\sum_i|\lambda^i|=(r+1)(l+1)$ determine both 
a critical level  CB bundle $\mathbb{V}(\sL_{r+1}, \lambda^{\bullet}, l)$  and a
GW divisor $I^{1,\Gr_{r,r+l}}_{1, \lambda^{\bullet}}$.  We  will show both divisors intersect all curves in the same degree. 

The $F$-curves, described in Definition~\ref{FCurves}, span the vector space of $1$-cycles, so it suffices to show that the intersections of the divisors with all $F$-curves are the same.    
An $F$-curve is indexed by a decomposition $\{1, \ldots, n\} = N_1 \cup \cdots \cup N_4$.
We write $\lambda(N_j)=\{\lambda^i : i \in N_j\}$. Recall that we write $(\mu^j)^\vee$ for the partition dual to $\mu^j$ given by taking the complement of   $\mu^j$  in a box of size $r\times l$ (pictured on the left of Figure~\ref{dualfig}).

 By \cite[Prop~2.2]{BG}, the degree of the intersection of an $F$-curve $F_{N_1,N_2,N_3,N_4}$ with the
GW divisor $I^{1,\Gr_{r,r+l}}_{1, \lambda^{\bullet}}$ is given by the formula
 \begin{equation}\label{GWE1}
 I^{1,\Gr_{r,r+l}}_{1, \lambda^{\bullet}}\cdot F_{N_1,N_2,N_3,N_4} 
 = \sum I^{1,\Gr_{r,r+l}}_{1-\sum_{j=1}^4d^j, \mu^{\bullet}} \prod_{j=1}^4I^{0,\Gr_{r,r+l}}_{d^j, \lambda(N_j)\cup (\mu^j)^\vee},\end{equation}
summing over 4-tuples of integers $d^{\bullet}=(d^1,\cdots,d^4)$ and $4$-tuples of partitions $\mu^{\bullet}=(\mu^1,\ldots,\mu^4)$ for $\Gr_{r,r+l}$. 
Note that we must have $1-\sum_{i}d^j\ge 0$, so $d_i\le 1$.  Furthermore, $I_{0, \mu^\bullet}^{1, \Gr_{r,r+l}}=0$. Hence, to have a nonzero summand, we may assume $d^j=0$ for all $j$.
Also, for $I^{0,\Gr_{r,r+l}}_{0, \lambda(N_j)\cup (\mu^j)^\vee}$ to be non-zero, we must have $ \sum_{i \in N_j}|\lambda^i|+|(\mu^j)^\vee| =rl$ or equivalently, $|\mu^j| = \sum_{i \in N_j} |\lambda^i|$.
Thus, the intersection of $F_{N_1,N_2,N_3,N_4}$ with $I^{1,\Gr_{r,r+l}}_{1, \lambda^{\bullet}}$
is given by 
\begin{equation}\label{GWE}
 I^{1,\Gr_{r,r+l}}_{1, \lambda^{\bullet}}\cdot F_{N_1,N_2,N_3,N_4} 
 = \sum_{\mu^{\bullet}} I^{1,\Gr_{r,r+l}}_{1, 
 \mu^{\bullet}} \prod_{j=1}^4I^{0,\Gr_{r,r+l}}_{0, \lambda(N_j)\cup (\mu^j)^\vee},\end{equation}
where our sum ranges over partitions $\mu^{\bullet}=\{\mu^j\}_{j=1}^4$ for
$\Gr_{r,r+l}$ satisfying 
\begin{equation}\label{GWData}
|\mu^j| = \sum_{i \in N_j}|\lambda^i| \text{ \ \  for $j = 1,2,3,4$}.
\end{equation}

The intersection of $F_{N_1,N_2,N_3,N_4}$  with $c_1(\mathbb{V}(\sL_{r+1},\lambda^{\bullet}, l))$ is given by the following formula (see Lemma~\ref{CLELemma}):
\begin{multline}\label{CLEP}  c_1(\mathbb{V}(\sL_{r+1},\lambda^{\bullet}, l)) \cdot F_{N_1,N_2,N_3,N_4} \\
 = \sum_{\nu^{\bullet}} c_1(\mathbb{V}(\sL_{r+1},\nu^{\bullet}, l)) \ 
  \prod_{1\le j \le 4}{\rm{Rk}}(\mathbb{V}(\sL_{r+1}, \lambda(N_j) \cup (\nu^j)^*, l)),\end{multline}
where one sums over $4$-tuples of partitions $\nu^{\bullet}=\{\nu^j\}_{j=1}^4$ of $\Gr_{r,r+l}$,  and   $(\nu^j)^*$ is the complement of $\nu^j$ in the rectangle of size $(r+1)\times \nu^j_1$. This is a slightly different notion of dual, pictured on the right of Figure~\ref{dualfig}.
\begin{figure}[h!]
\centering
\begin{tikzpicture}[scale=.3]
\draw[fill = blue!40] (0, -7) -- (0, -5) -- (3, -5) -- (3, -2) -- (4, -2) -- (4, -1) -- (5, -1) -- (5, 0) -- (8, 0) -- (8, -7) -- cycle;
\node[scale=.8] at (6.5, -4.5) {$\nu^\vee$};

\draw[white] (0, 0) rectangle (0, -8);
\draw[thick] (0, 0) rectangle (8, -7);
\draw (0, -5) -- (3, -5) -- (3, -2) -- (4, -2) -- (4, -1) -- (5, -1) -- (5, 0);
\node[scale=.8] at (1.5, -1.5) {$\nu$};

\node[scale=.8] at (4, 1) {$l$};
\draw [->] (3.5, 1) -- (0, 1);
\draw [->] (4.5, 1) -- (8, 1);
\node[scale=.8, rotate = 90] at (-1, -3.5) {$r$};
\draw [->] (-1, -4) -- (-1, -7);
\draw [<-] (-1, 0) -- (-1, -3);
\end{tikzpicture}
\hspace{1in}
\begin{tikzpicture}[scale=.3]
\draw[fill = red!40] (0, -8) -- (0, -5) -- (3, -5) -- (3, -2) -- (4, -2) -- (4, -1) -- (5, -1) -- (5, -8) -- cycle;
\node[scale=.8] at (4, -6) {$\nu^*$};

\draw[thick] (0, 0) rectangle (8, -7);
\draw (0, -5) -- (3, -5) -- (3, -2) -- (4, -2) -- (4, -1) -- (5, -1) -- (5, 0);
\node[scale=.8] at (1.5, -1.5) {$\nu$};
\draw[dashed, thick, red] (0, 0) rectangle (5, -8);

\node[scale=.8] at (4, 1) {$l$};
\draw [->] (3.5, 1) -- (0, 1);
\draw [->] (4.5, 1) -- (8, 1);
\node[scale=.8,rotate = 90] at (-1, -3.5-.5) {$r+1$};
\draw [->] (-1, -4.6-.5-.3) -- (-1, -8);
\draw [<-] (-1, 0) -- (-1, -2.4-.5+.3);
\end{tikzpicture}
\caption{Two notions of duals}
\label{dualfig}
\end{figure}
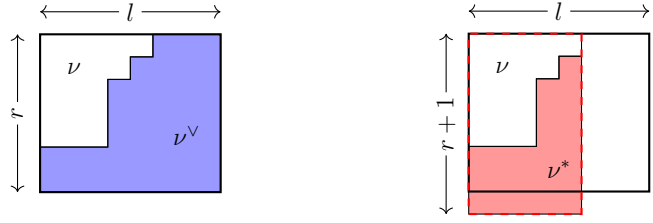

By Lemma~\ref{CLELemma}, the term for $\nu^\bullet$ in \eqref{CLEP} is zero unless
\begin{equation}\label{CLData}
|\nu^i|=\sum_{i\in N_j}|\lambda^i | \text{ \ \  for $j = 1,2,3,4$}.
\end{equation}
Thus, the nonzero terms of \eqref{GWE} and \eqref{CLEP} are both indexed by $4$-tuples of partitions satisfying \eqref{GWData} (equivalently \eqref{CLData}).

We will prove that \eqref{GWE} and  \eqref{CLEP} are equal if we show that
\begin{enumerate}
\item[(a)] $I^{1,\Gr_{r,r+l}}_{1, \mu^{\bullet}} 
=  \ c_1(\mathbb{V}(\sL_{r+1},\mu^{\bullet},l))$; and
\item[(b)]  $I^{0,\Gr_{r,r+l}}_{0,\lambda(N_j) \cup (\mu^j)^\vee}={\rm{Rk}}(\mathbb{V}(\sL_{r+1}, \lambda(N_j) \cup (\mu^j)^*, l))$ $\forall$ $1\le j\le 4$.
\end{enumerate}
By (the cohomological form of) Witten's Dictionary in \S\ref{WittenD},  since 
\[\sum_{i\in N_j}|\lambda^i| +|(\mu^j)^*|=(r+1)\mu^j_1=(r+1)(l+s), \ \mbox{ for } \  s \le 0,\]  
setting $\lambda(N_j) = \{\gamma^1,\ldots, \gamma^k\}$,  the rank of the vector bundle $\mathbb{V}(\sL_{r+1}, \lambda(N_j) \cup (\mu^j)^*, l)$  is 
 equal to the intersection number 
 \begin{equation} \label{rk}
 \sigma_{\gamma^1} \cdot \sigma_{\gamma^2}  \cdots  \sigma_{\gamma^k} \cdot \sigma_{(\mu^j)^*}\in \H^*\Gr_{r+1,r+1+l+s}.
 \end{equation}
Since $\mu^j$ has width $\mu_1^j=l+s$, $(\mu^j)^*$ is the complement of $\mu^j$ in an $(r+1) \times (l + s)$ rectangle. Therefore, the quantity in \eqref{rk} is equal to the classical generalized Littlewood-Richardson coefficient $c_{\gamma^\bullet}^{\mu^j}$, which can be computed in any Grassmannian where $\mu^j$ lies, in particular in $\Gr_{r,r+l}$. From \eqref{eq:classical}, the rank is therefore also equal to  $I^{0,\Gr_{r,r+l}}_{0,\lambda(N_j) \cup (\mu^j)^\vee}$, establishing (b) above. Part  (a) holds by the assumption of the statement of the theorem.

\begin{remark}
For fixed  $(r, l, \lambda^{\bullet})$, Theorem~\ref{thm:ConjReduction} reduces the GW $\equiv$ CB Conjecture to a finite computation.
For small $r$ and $l$, this is feasible with a computer and allows us to establish several new cases of the conjecture (see \S\ref{KnownCases}).
\end{remark}
\subsection{Reduction of Theorem~\ref{GWCB} to the case \texorpdfstring{$n=4$}{n=4}}
Theorem~\ref{GWCB} shows the GW $\equiv$ CB Conjecture holds for all partitions satisfying the column condition (see Definition~\ref{ColumnDef}). In the following, we show that it suffices to prove Theorem~\ref{GWCB} in the case $n=4$.

\begin{proposition}\label{prop:ReductionColumnId}
The GW $\equiv$ CB Conjecture holds for all $n$-tuples of partitions satisfying the column condition if the GW $\equiv$ CB Conjecture holds for all $4$-tuples of partitions satisfying the column condition.
\end{proposition}

\begin{proof}
For partitions $\lambda,\mu,\nu$, consider the Littlewood-Richardson coefficient $c_{\lambda, \mu}^\nu$.  We use the following basic fact from Schubert calculus:
\begin{equation} \label{prod_of_short}
\text{If $c_{\lambda ,\mu}^\nu \neq 0$, then $\# \nu \leq \#\lambda + \#\mu$}.
\end{equation}
We also use the ideas in the proof of Theorem~\ref{thm:ConjReduction}. In particular, we gave a correspondence between the non-zero terms of the sums in \eqref{GWE} and \eqref{CLEP}.
In a factor of a non-zero term
\[I^{0,\Gr_{r,r+l}}_{0, \lambda(N_j) \cup (\mu^j)^\vee} = \prod_{i \in N_j} \sigma_{\lambda^i} \cdot \sigma_{(\mu^j)^\vee} = {\rm{Rk}}\mathbb{V}(\sL_{r+1}, \{{\lambda^i}\}_{i \in N_j}\cup {(\mu^j)^*},l),\]
the partition $\mu^j$ must appear with non-zero coefficient in the product $\prod_{i \in N_j} \lambda^i$.

 Now suppose that $(\lambda^1, \ldots, \lambda^n)$ satisfy the column condition.
  By  \eqref{prod_of_short}, the $\nu$ with non-zero coefficients in $\prod_{i \in N_j} \lambda^i$ have $\#\nu \leq \sum_{i \in N_j} \# \lambda^i$. Hence, the term for $(\mu^1, \ldots, \mu^4)$ in \eqref{GWE} and \eqref{CLEP} is zero unless 
 \begin{equation} \label{boundmus}
 \# \mu^j \leq \sum_{i \in N_j} \# \lambda^i \qquad \text{for all $j$.}
 \end{equation}
 In particular, we actually need only show $I^{1,\Gr_{r,r+l}}_{1, \mu^{\bullet}} 
=  \ c_1(\mathbb{V}(\sL_{r+1},\mu^{\bullet},l))$ for $(\mu^1, \ldots, \mu^4)$ satisfying \eqref{boundmus}.
If our original collection $(\lambda^1, \ldots, \lambda^n)$ satisfies the column condition, then $(\mu^1, \ldots, \mu^4)$ satisfying \eqref{boundmus} satisfies
\[\sum_{j=1}^4 \#\mu^j \leq \sum_{j=1}^4 \sum_{i \in N_j} \# \lambda^i = \sum_{i=1}^n \#\lambda^i \leq 2(r+1),\]
which is the column condition for the $4$-tuple $(\mu^1, \ldots, \mu^4)$.
\end{proof}

\section{Connection to two-step flag varieties and GW-invariants for \texorpdfstring{$d=1$}{d=1}}\label{sec:lines} 

In this section, we  review the ``quantum-equals-classical" result  of \cite{BKT} which computes 3-pointed Gromov-Witten invariants on $\Gr_{r,r+l}$ as intersection numbers on a two-step flag variety $\Fl_{r-d,r+d;r+l}$ of nested subspaces $V_{r-d}\subset V_{r+d}$ in an $r+l$-dimensional vector space, with $\dim V_i=i$. We  extend this relationship in the case $d=1$ to $n$-pointed Gromov-Witten invariants. As a consequence,  we can compute dimension-$0$ GW classes on a two-step 
variety. When $n = 4$, the dimension-$0$ GW classes are divisors and this is a key step for our GW $\equiv$ CB result.

To state \cite[Corollary~1]{BKT}, we use the following terminology. 
As  discussed in \S\ref{sec:LR}, the basis of Schubert classes $\sigma_\lambda$ for $\Gr_{r,r+l}$ is indexed by partitions
$\lambda$ contained in an $r\times l$ rectangle. Such a partition $\lambda$ can be uniquely identified with a permutation $w_\lambda\in S_{r+l}$ by defining $w_\lambda(i)=\lambda_{r-i+1}+i$ for $1\leq i\leq r$ and then ordering the values $w_\lambda(r+1)<\cdots<w_\lambda(r+l)$. Note that  $w_\lambda(i)<w_\lambda(i+1)$ for  $i\neq r$, i.e.\  $w_\lambda$ is a  Grassmann permutation with only possible descent at $r$ .

For $1\leq d\leq \min\{r,l\}$ and $\lambda\subseteq (l^r)$, consider the permutation obtained from $w_\lambda$ by sorting the values $w_\lambda(r-d+1),\dots,w_\lambda(r+d)$ in increasing order. By construction, this has descents at most at $r-d$ and $r+d$, and so  corresponds to a Schubert class $\sigma_\lambda^{(d)}$ on the flag variety $\Fl_{r-d,r+d;r+l}$ (for more details on two-step flag varieties, including the Schubert basis as well as an alternative basis of classes on two-step flag varieties indexed by pairs of partitions, following \cite{gukalashnikov}, see \S\ref{a:twostep}).

By \cite[Corollary~1]{BKT}, for  partitions $\lambda^1,\lambda^2,\lambda^3\subseteq (l^r)$ satisfying $|\lambda^1|+|\lambda^2|+|\lambda^3|=rl+(r+l)d$,  we have:
\begin{equation}\label{eq:twostep}
I_d(\sigma_{\lambda^1},\sigma_{\lambda^2},\sigma_{\lambda^3})=  \int_{\Fl_{r-d, r+d; r+l}}\sigma_{\lambda^1}^{(d)} \cdot\sigma_{\lambda^2}^{(d)} \cdot\sigma_{\lambda^3}^{(d)}.
\end{equation}

We will show that the $n$-pointed Gromov-Witten invariant $I_1(\sigma_{\lambda^1},\dots,\sigma_{\lambda^n})$ can be computed using classical Schubert calculus on  $\Fl_{r-1,r+1;r+l}$.
 
\begin{proposition}\label{prop:lines} Consider an $n$-tuple of partitions $\lambda^\bullet=(\lambda^1,\dots,\lambda^n)$ contained in an $r\times l$ rectangle, satisfying $c\coloneqq\sum_{i\in [n]}|\lambda^i|-r-l-rl=n-3$, and let $\sigma_{\lambda^1}^{(1)},\cdots,\sigma_{\lambda^n}^{(1)}$ be the associated classes in $\H^*\Fl_{r-1,r+1;r+l}$.
Then
\[  I^{n-3,\Gr_{r,r+l}}_{1,\lambda^\bullet} = I_1(\sigma_{\lambda^1}, \ldots, \sigma_{\lambda^n}) =  \int_{\Fl_{r-1, r+1; r+l}}\prod_{i=1}^{n}\sigma_{\lambda^i}^{(1)}. \]
When $n=3$, this recovers \eqref{eq:twostep} for $d=1$. When $n=4$, this computes the GW-divisor $I^{1,\Gr_{r,r+l}}_{1,\lambda^1,\dots,\lambda^4}$.
\end{proposition}

\begin{remark}
The second equality in Proposition~\ref{prop:lines} doesn't require \(n\geq3\).
\end{remark}

Before proving Proposition~\ref{prop:lines}, we need the following lemma, which is a special case of the main theorem in \cite{martin}. See also \cite[Proposition~4.1.5]{KV07} for the projective space case. We give a simple alternative proof for our case.

\begin{lemma}\label{Geometric} Let $\lambda^\bullet=(\lambda^1,\dots,\lambda^n)$ be an $n$-tuple of partitions contained in an $r\times l$ rectangle  satisfying $\sum_{i\in [n]}|\lambda^i|-r-l-rl=n-3$. Then
the $n$-pointed Gromov-Witten invariant $I_1(\sigma_{\lambda^1},\dots,\sigma_{\lambda^n})$ is equal to the number of lines in $\Gr_{r,r+l}$ that meet $g_1X_{\lambda^1}, \cdots, g_nX_{\lambda^n}$, where $g_iX_{\lambda^i}$ are general translates of the associated Schubert varieties in  $\Gr_{r,r+l}$.\end{lemma}

\begin{proof}
  First note that if $L$ is a line in $\Gr_{r,r+l}$ and $X$ a Schubert variety in $\Gr_{r,r+l}$, then $L\cap X$ is $L$, one point, or empty. To see this, we have $L=\{\Sigma\in \Gr_{r,r+l}: K\subset\Sigma\subset S\}$ for some subspaces $K, S$ in $\C^{r+l}$ with $\dim K=r-1$, $\dim S=r+1$. Each Schubert variety is an intersection of Schubert varieties of the form $\{\Sigma\in \Gr_{r,r+l}: \text{dim}(\Sigma\cap F)\geq j\}$ for some subspace $F$ of $\C^{r+l}$. Without loss of generality, assume $X$ is of this form. Suppose $L\cap X$ contains two distinct points.  Then either $\dim(K\cap F)\geq j$, or $\dim(K\cap F)=j-1$ and $\dim(S\cap F)=j+1$. Either way, $L$ is contained in $X$.

Let $L$ be a line in $\Gr_{r,r+l}$ that meets all $g_iX_{\lambda^i}$. By the above, each $g_iX_{\lambda^i}$ contains either the entire $L$ or exactly one point in $L$. On the other hand, the intersections $g_iX_{\lambda^i}\cap L$ must be disjoint, because otherwise we can construct a map in the boundary of $\M_{0,n}(\Gr_{r,r+l},1)$ with image $L$, but the intersection $ev_1^{-1}(g_1X_{\lambda^1})\cap \cdots\cap ev_n^{-1}(g_nX_{\lambda^n})$ is supported on $\op{M}_{0,n}(\Gr_{r,r+l},1)$ \cite[Lemma~14]{FP}. Therefore, each $g_iX_{\lambda^i}$ must meet $L$ at a distinct point in $L$. Note that a degree $1$ map from $\mathbb{P}^1$ to $\Gr_{r,r+l}$ is an isomorphism onto its image. Since the choice of marked points exists and is unique, each $L$ uniquely determines a map in $ev_1^{-1}(g_1X_{\lambda^1})\cap \cdots\cap ev_n^{-1}(g_nX_{\lambda^n})$ and vice versa.
\end{proof}

\begin{proof}[Proof of Proposition~\ref{prop:lines}]
The first equality is just the observation in \eqref{eq:GWinvariant} that the degree of the dimension-0 GW-class is equal to the Gromov-Witten invariant.

Using Lemma~\ref{Geometric}, we can compute $I_1(\sigma_{\lambda^1}, \ldots, \sigma_{\lambda^n})$ using intersection theory on a two-step flag variety. Consider the diagram
\begin{center}
\begin{tikzcd}
\Fl_{r - 1, r, r+1; r+l} \arrow{r}{p} \arrow{d}[swap]{q} & \Gr_{r, r+l} \\
\Fl_{r-1, r+1; r+l}.
\end{tikzcd}
\end{center}
For a Schubert variety $X_{\lambda}$ in $\Gr_{r, r+l}$, let 
\[X_{\lambda}^{(1)}\coloneqq q(p^{-1}(X_\lambda))=\{(A,B)\in \Fl_{r-1, r+1; r+l}\vert\exists V\in X_{\lambda} \text{ with } A\subset V\subset B\}\]
be the Schubert variety in $\Fl_{r-1, r+1; r+l}$ considered in \cite[\S 2.2]{BKT} and $\sigma_{\lambda}^{(1)}$ its class. When one of the \(X_{\lambda^i}\) is the entire \(\Gr_{r,r+l}\), Proposition~\ref{prop:lines} holds because all three numbers are \(0\). Now assume each \(X_{\lambda^i}\) has positive codimension. 

When \(X_{\lambda}\) has positive codimension, it is contained in a Schubert divisor, which intersects a general line at one point.
Therefore, a general line meeting $X_{\lambda}$ meets it in one point and the map $q$ sends $p^{-1}(X_\lambda)$ generically one-to-one onto its image. It follows that 
\begin{equation}\label{eq:sigmalambda}
\sigma_{\lambda}^{(1)} = [q(p^{-1}(X_{\lambda})] = q_*p^*[X_{\lambda}] = q_*p^*\sigma_{\lambda}. \end{equation}

Moreover, $p^{-1}(X_\lambda)$ is the space of pairs $(L, V)$ where $L$ is a line in $\Gr_{r, r+l}$ and and $V \in L \cap X_\lambda$ is a point. Hence, $X_\lambda^{(1)}$ is the subvariety of lines $L$ on $\Gr_{r, r+l}$ that meet $X_\lambda$, so Lemma~\ref{Geometric} shows that
\begin{equation} \label{int-on-Z} \int_{\overline{M}_{0,n}(\Gr_{r,r+l},1)}\prod_{i=1}^n ev_i^*\sigma_{\lambda^i} = \int_{\Fl_{r-1, r+1; r+l}}\prod_{i=1}^{n}\sigma_{\lambda^i}^{(1)}.
\end{equation}
Since the left-hand side is exactly the $n$-pointed Gromov-Witten invariant $I_1(\sigma_{\lambda^1}, \ldots, \sigma_{\lambda^n})$ \cite{FP}, this concludes the proof.
\end{proof}

  \section{The  GW \texorpdfstring{$\equiv$}{=} CB Conjecture for a class of partitions}\label{ElanaIdea}

 By the previous section, we have turned the problem of computing degrees of GW divisors into one of computing intersections of certain classes on $\Fl_{r-1,r+1;r+l}$.
When  partitions $(\lambda^1, \ldots, \lambda^4)$ satisfy the column condition (Definition~\ref{ColumnDef}), we show that this product can be expressed in terms of intersection products on two Grassmannians. The main result of this section is the following.

\begin{proposition}\label{pro:conjspecialcase} Let $(\lambda^1, \ldots, \lambda^4)$ be partitions indexing Schubert classes in $\Gr_{r,r+l}$. Suppose $\sum_i |\lambda^i| =(r+1)(l+1)$ and  $\sum_i \#\lambda^i \le 2(r+1)$. Then
\[I^{1,\Gr_{r,r+l}}_{1,\lambda^1,\ldots,\lambda^4}\equiv c_1(\mathbb{V}(\sL_{r+1},\lambda^{\bullet},l)).\]
If the strict inequality $\sum_i \#\lambda^i <2(r+1)$ holds, both divisors are $0$.
\end{proposition}

Combining Propositions \ref{prop:ReductionColumnId} and \ref{pro:conjspecialcase} gives Theorem~\ref{GWCB}.
Proposition~\ref{pro:conjspecialcase} is proved in \S\ref{Proof:pro:conjspecialcase}.

\subsection{The Gromov-Witten side}\label{GWSide}

We use Proposition~\ref{pro:comparison} to compute degree 1, 4-pointed Gromov-Witten invariants on $\Gr_{r,r+l}$ via Schubert calculus on the two-step flag variety $\Fl_{r-1,r+1;r+l}$. 

We first explain how to describe Schubert classes $\sigma_{\lambda}^{(1)}$ on $\Fl_{r-1,r+1;r+l}$  using   pairs of partitions, following the notation of \S\ref{a:twostep}.

\begin{definition} \label{assoc-pair}
Given a partition $\lambda\subseteq (l^r)$,  define $\alpha$ to be a single column of height $\#\lambda-1$, i.e.\  $\alpha=(1^{\#\lambda-1})$. We picture $\alpha$ as the first column of $\lambda$ minus a box. Define $\beta$ to be the partition obtained by removing the first column of $\lambda$.  We view $\alpha$ as lying in an $(r-1)\times 2$ rectangle and $\beta$ in an $(l-1)\times r$ rectangle.
\end{definition}

\begin{example}
If $\lambda = (4, 4, 2, 1)$, then $\alpha = (1, 1, 1)$ and $\beta = (3, 3, 1)$.
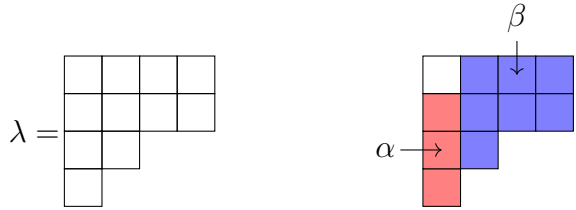
\begin{center}
\begin{tikzpicture}
\draw(0, -.5) rectangle (.5, 0);
\draw (0, 0) rectangle (.5, .5);
\draw (.5, 0) rectangle (1, .5);
\draw(0, .5) rectangle (.5, 1);
\draw (.5, .5) rectangle (1, 1);
\draw (1, .5) rectangle (1.5, 1);
\draw (1.5, .5) rectangle (2, 1);
\draw (0, 1) rectangle (.5, 1.5);
\draw (.5, 1) rectangle (1, 1.5);
\draw (1, 1) rectangle (1.5, 1.5);
\draw (1.5, 1) rectangle (2, 1.5);
\node at (-.4, .5) {$\lambda =$};
\end{tikzpicture}
\hspace{.7in}
\begin{tikzpicture}
\draw[fill=red!50] (0, -.5) rectangle (.5, 0);
\draw[fill=red!50] (0, 0) rectangle (.5, .5);
\draw [fill=blue!50] (.5, 0) rectangle (1, .5);
\draw[fill=red!50] (0, .5) rectangle (.5, 1);
\draw[fill=blue!50] (.5, .5) rectangle (1, 1);
\draw[fill=blue!50] (1, .5) rectangle (1.5, 1);
\draw[fill=blue!50] (1.5, .5) rectangle (2, 1);
\draw (0, 1) rectangle (.5, 1.5);
\draw[fill=blue!50] (.5, 1) rectangle (1, 1.5);
\draw[fill=blue!50] (1, 1) rectangle (1.5, 1.5);
\draw[fill=blue!50] (1.5, 1) rectangle (2, 1.5);
\node at (-.5, 0.25) {$\alpha$};
\draw[->] (-.3, .25) -- (.25, .25);
\node at (1.25, 2) {$\beta$};
\draw[->] (1.25, 1.7) -- (1.25, 1.25);
\end{tikzpicture}
\end{center}
\end{example}

\begin{proposition}[Gromov-Witten divisor identity]\label{pro:identity} Let $(\lambda^1,\ldots, \lambda^4)$ be partitions defining Schubert classes on the Grassmannian $\Gr_{r,r+l}$. 
Let $(\alpha^i, \beta^i)$ be the associated pair of partitions for $\lambda^i$ as in Definition~\ref{assoc-pair}.
 Suppose $\sum_i |\lambda^i| =(r+1)(l+1)$ and
$\sum_i \#\lambda^i \leq 2(r+1)$.
Then
\begin{equation} \label{propeq}
I^{1,\Gr_{r,r+l}}_{1,{\lambda^1,\dots,\lambda^4}} = I^{0,\Gr_{r-1,r+1}}_{0,{\alpha^1,\dots,\alpha^4}}I^{0,\Gr_{r+1,r+l}}_{0,{\beta^1,\dots,\beta^4}}.
\end{equation}
If $\sum_i \#\lambda^i < 2(r+1)$, then $I^{1,\Gr_{r,r+l}}_{1,{\lambda^1,\dots,\lambda^4}} = 0$.
\end{proposition}

\begin{proof} Since $ |\alpha^i| = \#\lambda^i -1$, we have $\sum_i |\alpha^i| \leq 2(r-1)$ and $\sum_i |\beta^i| = (r+1)(l-1)$. The result \eqref{propeq} now follows from Proposition~\ref{pro:comparison}, Proposition~\ref{prop:lines} with $n = 4$, and \eqref{eq:classical}. If $\sum_i \#\lambda^i < 2(r+1)$, then $\sum_i |\alpha^i| < 2(r-1)$, so $I^{0,\Gr_{r-1,r+1}}_{0,{\alpha^1,\dots,\alpha^4}} = 0$ and hence, $I^{1,\Gr_{r,r+l}}_{1,{\lambda^1,\dots,\lambda^4}} = 0$ as well.
\end{proof}

\subsection{An analogous identity on the critical level CB-bundle side}\label{VerlindeSide}
The aim of this section is to establish an identity for critical level CB-bundles satisfying the column condition (Definition~\ref{ColumnDef}). We consider four-tuples $\{\lambda^i\}$ of partitions. We suppose $\overline{\alpha}^i$ is a column of the partition $\lambda^i$, so for each $i$, there is a partition $\beta^i$ that lives in an $r\times (l-1)$ box, and for which $\lambda^i=\overline{\alpha}^i +\beta^i$.

 Witten's dictionary is used to calculate the ranks of the vector bundles of coinvariants in type $A$ in terms of quantum cohomology.  By translating this 
 to a classical calculation via rim-hook removals (using Lemma~\ref{mainlemma}), we prove the following.
 
\begin{lemma}\label{newwitten}
Given  a collection  $\lambda^\bullet$ of $n$ partitions, each contained in an $r\times l$ rectangle that satisfies
$\sum_i |\lambda^i| =(r+1)(l+1)$, if  $\sum_i \#\lambda^i =(r+1)(2)$, then  the rank of $\mathbb{V}(\sL_{r+1}, \lambda^\bullet,l)$ 
on $\M_{0,n}$ is equal to a classical generalized Littlewood-Richardson coefficient:
 \[{\rm{Rk}} (\mathbb{V}(\sL_{r+1}, \lambda^{\bullet}, l)) = c_{\lambda^\bullet}^{(l^{r+1},1^{r+1})}.\]
 \end{lemma}
 \begin{proof}
By the formulation of Theorem~\ref{WD}  in \S\ref{WittenD}, the rank $R$
  is equal to the degree $s=1$ generalized quantum Littlewood-Richardson coefficient $c_{\lambda^\bullet,(l)}^{1,(l^{r+1})}$ on $\QH^*\Gr_{r+1,r+1+l}$. The result follows immediately from Lemma~\ref{lem:quantum-classical}.
\end{proof}

The following identity is analogous to Proposition~\ref{pro:identity}.

\begin{proposition}[Critical level divisor identity]\label{V:identity} Let $\lambda^\bullet=(\lambda^1,\dots,\lambda^n)$ be a collection of partitions inside an $r\times l$ rectangle. Suppose
$\sum_i |\lambda^i| =(r+1)(l+1)$  and $\sum_i \#\lambda^i =(r+1)(2).$
Then 
\[c_1(\mathbb{V}(\sL_{r+1},\lambda^{\bullet}, l)) = c_1(\mathbb{V}(\sL_{r+1}, \overline{\alpha}^{\bullet}, 1)){\rm{Rk}} \mathbb{V}(\sL_{r+1}, \beta^{\bullet},(l-1)),\]
where for each $1\leq i\leq n$,  $\lambda^i=\overline{\alpha}^i + \beta^i$, where $\overline{\alpha}^i$ is the first column of $\lambda^i$. 
\end{proposition}

\begin{proof} 
The first step in the proof is to show that 
\begin{equation} \label{rkid}
{\rm{Rk}}(\mathbb{V}(\sL_{r+1}, \lambda^{\bullet}, l))={\rm{Rk}}(\mathbb{V}(\sL_{r+1}, \beta^{\bullet}, (l-1))).
\end{equation}
Since $|\overline{\alpha}^i|=\#\lambda^i$, we have $\sum_i |\beta^i| = (r+1)(l-1)$, and so by Theorem~\ref{WD}, the right hand side is equal to the generalized Littlewood-Richardson coefficient $c_{\beta^\bullet}^{(l-1)^{r+1}}$. Applying Lemma~\ref{newwitten} to the left hand side, it suffices to show that
\[c^{(l^{r+1},1^{r+1)}}_{\lambda^{\bullet}}=c^{(l-1)^{r+1}}_{\beta^{\bullet}}.\]
Applying Lemma~\ref{lem:LR} to  $\lambda^\bullet$ gives this equality.

Having established the rank equality \eqref{rkid},  by \cite[Proposition~19]{BGMB},
\begin{multline}\label{Decomp}
c_1(\mathbb{V}(\sL_{r+1}, \lambda^{\bullet}, l))=c_1(\mathbb{V}(\sL_{r+1}, \overline{\alpha}^{\bullet}, 1)){\rm{Rk}}(\mathbb{V}(\sL_{r+1}, \beta^{\bullet}, (l-1)))\\
+
c_1(\mathbb{V}(\sL_{r+1}, \beta^{\bullet}, (l-1))){\rm{Rk}}( \mathbb{V}(\sL_{r+1}, \overline{\alpha}^{\bullet}, 1)).\end{multline}
We will show the second line of \eqref{Decomp} is zero.
Since $\sum_{i=1}^n|\beta^i|=(r+1)(l-1)$, recalling the definition 
from \S\ref{CL}, the critical level of  the pair $(\sL_{r+1}, \beta^{\bullet})$ is 
\[c(\sL_{r+1}, \beta^{\bullet})=(l-1)-1=l-2,\] and so 
as the level of $\mathbb{V}(\sL_{r+1}, \beta^{\bullet}, (l-1))$ is $l-1>l-2$,   by \cite[Thm~1.3]{BGMA}, we conclude that $c_1(\mathbb{V}(\sL_{r+1}, \beta^{\bullet}, (l-1)))=0$. 
In particular, \eqref{Decomp} becomes
\[c_1(\mathbb{V}(\sL_{r+1}, \lambda^{\bullet}, l))
=c_1(\mathbb{V}(\sL_{r+1}, \overline{\alpha}^{\bullet}, 1)){\rm{Rk}}(\mathbb{V}(\sL_{r+1}, \beta^{\bullet}, (l-1))).\]
The proposition follows. 
\end{proof}

\subsection{Proof of Proposition~\ref{pro:conjspecialcase} and Theorem~\ref{GWCB}}\label{Proof:pro:conjspecialcase}
\begin{proof}
By Proposition~\ref{prop:ReductionColumnId}, Proposition~\ref{pro:conjspecialcase} implies Theorem~\ref{GWCB}.
Thanks to the two identities (Propositions \ref{pro:identity} and \ref{V:identity}), to prove Proposition~\ref{pro:conjspecialcase}, it suffices to show that
\[{\rm{deg}}(\mathbb{V}(\sL_{r+1}, \overline{\alpha}^{\bullet}, 1)){\rm{Rk}} \mathbb{V}(\sL_{r+1}, \beta^{\bullet}, l-1)= I^{0,\Gr_{r-1,r+1}}_{0,{\alpha^1,\dots,\alpha^4}}I^{0,\Gr_{r+1,r+l}}_{0,{\beta^1,\dots,\beta^4}}.\]
We start on the CB side. From the proof of Proposition~\ref{V:identity}, it follows that
\[{\rm{Rk}}( \mathbb{V}(\sL_{r+1}, \beta^{\bullet}, l-1))=c^{(l-1)^{r+1}}_{\beta^{\bullet}}=I^{0,\Gr_{r+1,r+l}}_{0,{\beta^1,\dots,\beta^4}}.\]
By the (known) $l=1$ case of the conjecture, 
\[c_1(\mathbb{V}(\sL_{r+1}, \overline{\alpha}^{\bullet}, 1))=I^{1,\Gr_{r,r+1}}_{1,\overline{\alpha^1},\dots,\overline{\alpha^4}}.\]
Finally, we apply Proposition~\ref{pro:identity} to the $\overline{\alpha^i}$ to see that
\[I^{1,\Gr_{r,r+1}}_{1,\overline{\alpha^1},\dots,\overline{\alpha^4}}=I^{0,\Gr_{r-1,r+1}}_{0,{\alpha^1,\dots,\alpha^4}}.\]
The theorem now follows. 
\end{proof}

\subsection{Corollaries}
We expect that the propositions above will allow us to see unexpected behavior on both the critical level CB and GW sides. For example, the following proposition is surprising from the perspective of conformal blocks (see Remark~\ref{WhySurprising}). 
\begin{lemma}\label{lem:addrowV} Let $(\lambda^1,\ldots, \lambda^4)$ be partitions for $\Gr_{r,r+l}$, with $\# \lambda^1 \geq \cdots \geq \# \lambda^4$ and $\sum_{i}|\lambda^i|=(r+1)(l+1)$. 
Let $\mu^1$ be obtained from $\lambda^1$ by adding a maximal row, let $\mu^2$ be obtained from $\lambda^2$ by adding a single box at the end of the first column, and let $\mu^3=\lambda^3$, and $\mu^4=\lambda^4$. Then 
\[c_1(\mathbb{V}(\sL_{r+1}, \lambda^{\bullet}, l))=c_1(\mathbb{V}(\sL_{r+2}, \mu^{\bullet}, l)).\]
\end{lemma}

\begin{proof}
First note that \[|\mu^1|+|\mu^2|+|\lambda^3|+|\lambda^4|=(r+1)(l+1)+l+1=(r+2)(l+1),\] as by assumption $|\lambda^1|+|\lambda^2|+|\lambda^3|+|\lambda^4|=(r+1)(l+1)$. If $\mu^i$ corresponds to the pair of partitions $(\tilde{\alpha}^i,\tilde{\beta}^i)$, and $\lambda^i$ to $(\alpha^i,\beta^i)$, then 
\[|\tilde{\alpha}^1|+|\tilde{\alpha}^2|+|{\alpha}^3|+ |{\alpha}^4|=2+|{\alpha}^1|+|{\alpha}^2|+|{\alpha}^3|+ |{\alpha}^4|=2+2(r-1)=2r.\]
This shows that the partitions $(\mu^1,\mu^2,\lambda^3,\lambda^4)$ satisfy the conditions of Proposition~\ref{pro:conjspecialcase}. It therefore suffices to show this statement on the Gromov-Witten locus side. That is, we show that 
\[I^{1,\Gr_{r,r+l}}_{1,\lambda^1,\lambda^2,\lambda^3,\lambda^4}=I^{1,\Gr_{r+1,r+1+l}}_{1,\mu^1,\mu^2,\lambda^3,\lambda^4}.\]
By Proposition~\ref{pro:identity},
\[I^{1,\Gr_{r+1,r+1+l}}_{1,\mu^1,\mu^2,\lambda^3,\lambda^4}=I^{0,\Gr_{r,r+2}}_{0,{\tilde{\alpha}^1,\tilde{\alpha}^2,\alpha^3,\alpha^4}}I^{0,\Gr_{r+2,r+l+1}}_{0,{\tilde{\beta}^1,\tilde{\beta}^2,\beta^3,\beta^4}},\]
and
\[I^{1,\Gr_{r,r+l}}_{1,{\lambda^1,\dots,\lambda^4}} = I^{0,\Gr_{r-1,r+1}}_{0,{\alpha^1,\dots,\alpha^4}}I^{0,\Gr_{r+1,r+l}}_{0,{\beta^1,\dots,\beta^4}}.\]
The lemma will follow from showing that
\[ I^{0,\Gr_{r-1,r+1}}_{0,{\alpha^1,\dots,\alpha^4}}=I^{0,\Gr_{r,r+2}}_{0,{\tilde{\alpha}^1,\tilde{\alpha}^2,\alpha^3,\alpha^4}} \hspace{2mm}  \text{ and } \hspace{2mm} I^{0,\Gr_{r+1,r+l}}_{0,{\beta^1,\dots,\beta^4}}=I^{0,\Gr_{r+2,r+l+1}}_{0,{\tilde{\beta}^1,\tilde{\beta}^2,\beta^3,\beta^4}}.\]
Notice that $\tilde{\beta}^1$ is $\beta^1$ with an extra maximal row added, while $\tilde{\beta}^2=\beta^2$. The second equality thus follows easily from Schubert calculus. For the first, note that for $i=1,2$, $\tilde{\alpha}^i$ is obtained from $\alpha^i$ by adding an extra box at the end of the column (these are both columns of length 1). Choosing $\alpha^1$ and $\alpha^2$ to be the longest of the four columns ensures that $|\alpha^1|+|\alpha^2| \geq r-1$, and hence $|\tilde{\alpha}^1|+|\tilde{\alpha}^2| \geq r+1$. Every partition $\mu$ fitting into an $r \times 2$ box with $c^{\mu}_{\tilde{\alpha}^1 \tilde{\alpha}^2} \neq 0$ has at least one maximal width row. Removing this row identifies the product $\sigma_{\alpha^1} \sigma_{\alpha^2}$ in $\Gr_{r-1,r+1}$ with that of  $\sigma_{\tilde{\alpha}^1} \sigma_{\tilde{\alpha}^2}$ in $\Gr_{r,r+2}$. The desired equality follows. 
\end{proof}

\begin{remark}\label{WhySurprising}
The bundles in  Lemma~\ref{lem:addrowV}  are at the critical level, and so by \cite[Proposition~1.6]{BGMA} the 
assertion is equivalent to the statement 
$c_1(\mathbb{V}(\sL_{l+1}, \lambda^T_{\bullet}, r))=c_1(\mathbb{V}(\sL_{l+1},\mu^T_{\bullet}, r+1)).$
At first glance, one may think that this can be shown by using the additive identity  \cite[Proposition~19]{BGMB}, 
to decompose $c_1(\mathbb{V}(\sL_{l+1},\mu^T_{\bullet},r+1))$ into a sum of the first Chern class of a level $1$ bundle for $\sL_{l+1}$ and $c_1(\mathbb{V}(\sL_{l+1},\lambda^T_{\bullet}, r))$, and hope that the level one bundle has a vanishing first Chern class.  To apply \cite[Proposition~19]{BGMB}, among other things, one needs 
${\rm{Rk}}(\mathbb{V}(\sL_{l+1}, \lambda^T_{\bullet}, r))={\rm{Rk}}(\mathbb{V}(\sL_{l+1},\mu^T_{\bullet}, r+1))$, which is not always the case.  For example, if
$\lambda^1=(3,2),
 \lambda^2=(2, 1)$, and
$\lambda^3 = \lambda^4=(2, 2),$
so $\sum_i |\lambda_i|=16$ and $r=\ell=3$, then for $\mu^{1}=(3,3,2),  \mu^2=(2, 1, 1)$, and $\mu^3=\mu^4 = (2, 2)$,
 one can compute ${\rm{Rk}}(\mathbb{V}(\sL_{4}, (\lambda^T)^{\bullet}, 3))=4$, and
 ${\rm{Rk}}(\mathbb{V}(\sL_{4},(\mu^T)^{\bullet}, 4))=5$.  From this perspective,  Lemma~\ref{lem:addrowV} is surprising.
 \end{remark}

\section{The GW \texorpdfstring{$\equiv$}{=} CB Conjecture in examples and in other cases}\label{KnownCases}
For each fixed $(r, l)$, Theorem~\ref{thm:ConjReduction} reduces the conjecture to a finite computation. Namely, we must check that for every collection of $4$ partitions of the correct sizes, the degree of the critical level CB-divisor agrees with the degree of the GW divisor. The degree of the critical level CB-divisor can be computed using the Macaulay2 package conformalBlocks.
By Proposition~\ref{prop:lines} and \eqref{eq:sigmalambda}, the degree of the GW divisor is equal to the degree of the
product $\prod_{i=1}^{4} q_* p^* \sigma_{\lambda^i}$, which is also readily computable using Macaulay2. Using this, we verified the conjecture for small values of $(r, l)$, listed below.

\begin{proposition}\label{M2}
For all collections $\lambda^\bullet$ of \(4\) partitions,
the GW divisor $I^{1, \Gr_{r, r+l}}_{1, \lambda^\bullet} $ is numerically equivalent to the corresponding critical level CB-divisor for 
\[(r, l) = (2, 2), (2, 3), (2, 4), \ldots, (2, 11), (3, 3), (3, 4).\] 
\end{proposition}

Our proof of Theorem~\ref{GWCB} gives rise to a sufficient combinatorial criterion for the non-vanishing of GW/CB divisors.
\begin{proposition} \label{nonvanishing}
 The GW divisor and the CB divisor associated to $\lambda^\bullet$ are non-zero if there exists a decomposition $[n] = N_1 \cup \cdots \cup N_4$ and partitions $(\mu^1, \ldots, \mu^4)$ such that 
\begin{enumerate}
\item $\sigma_{\mu^j}$ appears with positive coefficient in $\prod_{i \in N_j} \sigma_{\lambda^i}$
\item the sum of the heights of the $\mu^j$ is equal to $2r+2$
\item the product of the $\sigma_{\beta^j}$ (where $\beta^j$ is obtained by removing the first column $\mu^j$) is non-zero in $\Gr_{r+1,r+l}$
\end{enumerate}
\end{proposition}
Note that condition (2) may be satisfied even if the original collection $\lambda^\bullet$ does not satisfy the column condition.  We give an example below.  It is often hard to know if appropriate $\mu^j$ exist.  However, by working backwards we can construct many examples where it is apparent that (1) -- (3) are satisfied.
\begin{proof}
 Condition (1) implies $\prod_{j=1}^4 I_{0,\lambda(N_j)\cup(\mu^j)^\vee}^{0,\Gr_{r,r+l}}$ is positive. Condition (2) and the Piere rules imply $I_{0,\alpha^1,\ldots,\alpha^4}^{0,\Gr_{r-1,r+1}}$ is positive. Condition (3) says $I_{0,\beta^1,\ldots\beta^4}^{0,\Gr_{r+1,r+l}}$ is positive. Thus, by Proposition~\ref{pro:identity}, we see $I_{1,\mu^1,\ldots,\mu^4}^{1,\Gr_{r,r+l}}$ is positive. In particular, 
\begin{equation} \label{oneterm}
I_{1,\mu^1,\ldots,\mu^4}^{1,\Gr_{r,r+l}}\prod_{j=1}^4 I_{0,\lambda(N_j)\cup(\mu^j)^\vee}^{0,\Gr_{r,r+l}}
\end{equation}
is positive. The term \eqref{oneterm} appears as a summand in \eqref{GWE} for the calculation of $I_{1,\lambda^\bullet}^{1,\Gr_{r,r+l}}\cdot F_{N_1,\ldots,N_4}$.  Since all summands in \eqref{GWE} are non-negative, it follows that $I_{1,\lambda^\bullet}^{1,\Gr_{r,r+l}}\cdot F_{N_1,\ldots,N_4}$ is positive. Hence, $I_{1,\lambda^\bullet}^{1,\Gr_{r,r+l}}$ is non-trivial.

Similarly, by Theorem~\ref{GWCB} in the case $n=4$, the term in \eqref{oneterm} is equal to the $\nu^\bullet=\mu^\bullet$ term in \eqref{CLEP}.  This shows that the CB divisor must also intersect this $F$-curve in positive degree.
\end{proof}

Using the conditions (1) -- (3) one can construct many examples that satisfy the  column condition and give nonzero GW/CB divisors.  We now describe one such infinite family. Fix $m\in\mathbb{Z}_{>0}$ and choose $l$ and $r$ so that $l$ is odd and $r+1$ is divisible by $2m$.  Take each of $n=(2r+2)/m$ partitions $\lambda^i$ to be a rectangle with height $m$ and width $(l+1)/2$.  Note that $\lambda^\bullet$ satisfies the column condition, as
\[\sum_{i=1}^n|\lambda^i|=\frac{2r+2}{m}\cdot m\cdot\frac{l+1}{2}=(r+1)(l+1)\]
and
\[\sum_{i=1}^n\#\lambda^i=\frac{2r+2}{m}\cdot m=2r+2.\]
Divide the set $\{1,\ldots,n\}$ evenly among $N_1,N_2,N_3,N_4$. For each $j=1, \ldots, 4$, let $\mu^j$ be the partition with height $(r+1)/2$ and width $(l+1)/2$.  
Notice that the union of $(r+1)/2m$ copies of $\lambda^i$ stacked vertically is the partition $\mu^j$ (indicated by bold lines in the figure below). Hence, condition (1) is readily seen to be satisfied by the Littlewood-Richardson rules. Condition (2) is also satisfied as
\[\sum_{j=1}^4 \#\mu^j = 4 \cdot \frac{r+1}{2} = 2r+2.\]
Finally, in condition (3), each $\beta^j$ is an $(r+1)/2$ by $(l-1)/2$ rectangle. These $4$ rectangles can be placed side by side to make an $r+1$ by $l - 1$ rectangle, so applying the Littlewood-Richardson rules, we see that condition (3) is also satisfied.

Pictured below are the partitions for this example when $m = 2, r = l = 11, n = 12$.  
 
The bold lines show how $\mu^j$ is a union of copies of $\lambda^i$.
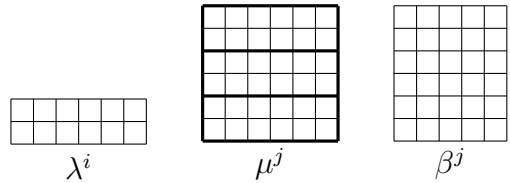
\begin{center}
\begin{tikzpicture}[scale = 0.3]
\draw (0, 0) -- (6, 0);
\draw (0, -1) -- (6, -1);
\draw (0, -2) -- (6, -2);
\draw (0, 0) -- (0, -2);
\draw (1, 0) -- (1, -2);
\draw (2, 0) -- (2, -2);
\draw (3, 0) -- (3, -2);
\draw (4, 0) -- (4, -2);
\draw (5, 0) -- (5, -2);
\draw (6, 0) -- (6, -2);
\node at (3, -3) {$\lambda^i$};
\end{tikzpicture}
\hspace{0.2in}
\begin{tikzpicture}[scale = 0.3]
\draw[very thick] (0, 0) -- (6, 0);
\draw (0, -1) -- (6, -1);
\draw[very thick] (0, -2) -- (6, -2);
\draw (0, -3) -- (6, -3);
\draw[very thick] (0, -4) -- (6, -4);
\draw (0, -5) -- (6, -5);
\draw[very thick] (0, -6) -- (6, -6);
\draw[very thick] (0, 0) -- (0, -6);
\draw (1, 0) -- (1, -6);
\draw (2, 0) -- (2, -6);
\draw (3, 0) -- (3, -6);
\draw (4, 0) -- (4, -6);
\draw (5, 0) -- (5, -6);
\draw[very thick] (6, 0) -- (6, -6);
\node at (3, -7) {$\mu^j$};
\end{tikzpicture}
\hspace{0.2in}
\begin{tikzpicture}[scale = 0.3]
\draw (0, 0) -- (5, 0);
\draw (0, -1) -- (5, -1);
\draw (0, -2) -- (5, -2);
\draw (0, -3) -- (5, -3);
\draw (0, -4) -- (5, -4);
\draw (0, -5) -- (5, -5);
\draw (0, -6) -- (5, -6);
\draw (0, 0) -- (0, -6);
\draw (1, 0) -- (1, -6);
\draw (2, 0) -- (2, -6);
\draw (3, 0) -- (3, -6);
\draw (4, 0) -- (4, -6);
\draw (5, 0) -- (5, -6);
\node at (2.5, -7) {$\beta^j$};
\end{tikzpicture}

\end{center}

In the example above, both $\mu^\bullet$ and $\lambda^\bullet$ satisfy the column condition.  Proposition~\ref{nonvanishing} can also be used to show the nonvanishing of divisors associated to $\lambda^\bullet$ not satisfying the column condition.  For instance, we can modify our example family above by ``cutting each $\lambda^i$ in half." Continuing the example with $r = l = 11$ above, we can take $n = 24$ and each $\lambda^i$ to be $(3,3)$.  Then $\sum_{i=1}^n \#\lambda^i = 24\cdot 2 = 48 > 24=2r+2$.  Nevertheless, criteria (1) -- (3) are still satisfied for $\{1,\ldots,24\}$ divided evenly among $N_1,N_2,N_3,N_4$ and each $\mu^j$ equal to $(6,6,6,6,6,6)$.

We end this section with one more family of examples, which generalizes to $\sL_{r+1}$ an example considered in \cite[\S 5]{BGMB} for $\sL_{2}$.  Take $\lambda^1 = \lambda^2 = (1), \lambda^3 = (l, 1^{r-1}), \lambda^4 = (l^r)$. Then $\sum_i |\lambda^i| = (r+1)(l+1)$.
Here is a picture when $r = 4, l = 5$:
\[\lambda^1 = \lambda^2 = {\tiny \ydiagram{1}}, \quad \lambda^3 = {\tiny \ydiagram{5,1,1,1}}, \quad \lambda^4 = {\tiny \ydiagram{5,5,5,5}}. \]
Then $\lambda^\bullet$ satisfies the column identity so Theorem~\ref{GWCB} says $I^{1\Gr_{r,r+l}}_{1, \lambda^\bullet} \equiv c_1(\mathbb{V}(\sL_{r+1}, \lambda^{\bullet}, l))$. Using Proposition~\ref{pro:identity}, one can compute directly that all divisors in this family have degree $1$.

 \appendix
 
 \section{Schubert calculus and two-step flag varieties}\label{a:LR}
 In this section, we collect some useful facts about Littlewood-Richardson coefficients and Schubert calculus. We also describe a basis of  the cohomology ring of  two-step flag varieties  in terms of pairs of partitions, following \cite{gukalashnikov}, and discuss its relation to the basis of Schubert classes.
 
 \subsection{Factorization of generalized Littlewood-Richardson coefficients}
 
We first give a technical result about generalized Littlewood-Richardson coefficients, which is a special case of the factorization of Littlewood-Richardson coefficients. Recall  the generalized Littlewood-Richardson coefficient $c^{\lambda}_{\lambda^1 \cdots \lambda^n}$ denotes the coefficient of $\lambda$ in the product of the Schur polynomials associated to the $\lambda^i$. 
\begin{lemma}\label{lem:LR} Let $(\lambda^1,\dots,\lambda^n)$ be  $n$ partitions, and suppose that $c^{\nu}_{\lambda^1 \cdots \lambda^n} >0$. Assume that $\#\nu=\sum_{i=1}^n \# \lambda^i$. Let $\hat{\nu}$ (respectively $\hat{\lambda}^i$) denote the partition obtained from $\nu$ (respectively $\lambda^i$) by removing the first column.
Then $c^{\nu}_{\lambda^1 \dots \lambda^n}=c^{\hat{\nu}}_{\hat{\lambda}^1 \dots \hat{\lambda}^n}.$
\end{lemma}
\begin{proof}
Triples of partitions $(\gamma,\delta,\rho)$ with non-zero Littlewood-Richardson coefficients lie in a cone cut out by Horn equalities. One (transposed) example of such an inequality is that $\#\gamma \leq \#\delta+\#\rho$.  Littlewood-Richardson coefficients on the boundary of the cone satisfy factorization properties, as shown in Theorem 1.4 of \cite{factorizationpaper}. After transposing, a special case of this theorem is the statement that 
\begin{equation} \label{basecase}
c^{\gamma}_{\delta\rho}=c^{\hat{\gamma}}_{\hat{\delta} \hat{\rho}} c^{(\#\gamma)}_{(\#\delta) (\#\rho)}=c^{\hat{\gamma}}_{\hat{\delta} \hat{\rho}}.
\end{equation}
The proof proceeds by induction. The base case, when $n=2$, is \eqref{basecase}. Suppose the statement holds for $n-1$ partitions. Let $S$ denote the set of partitions $\mu$ such that $c^{\mu}_{\lambda^1 \cdots \lambda^{n-1}} >0$. By induction, there is a one-to-one correspondence between $S$ and the corresponding set 
$\hat{S}\coloneqq\{\eta \mid c^{{\eta}}_{\hat{\lambda}^1 \dots \hat{\lambda}^{n-1}}>0\}$
 for the $\hat{\lambda}^i$, given by taking $\mu \in S$ to $\hat{\mu} \in \hat{S}$. 
 Now
 \[c^{\nu}_{\lambda^1 \dots \lambda^n}=\sum_{\mu \in S} c^{\mu}_{\lambda^1 \cdots \lambda^{n-1}} c^{\nu}_{\mu \lambda^n}.\]
 The assumption holds for both factors in each summand, so 
 \[c^{\nu}_{\lambda^1 \dots \lambda^n}=\sum_{\mu \in S} c^{\hat{\mu}}_{\hat{\lambda}^1 \dots \hat{\lambda}^{n-1}} c^{\hat{\nu}}_{\hat{\mu}\hat{\lambda}^n}=c^{\hat{\nu}}_{\hat{\lambda}^1 \dots \hat{\lambda}^n}. \qedhere\]
\end{proof}

\begin{lemma}\label{lem:quantum-classical}
Consider a collection of partitions $\lambda^\bullet=(\lambda^1,\dots,\lambda^n)$ in an $r\times l$ rectangle such that $\sum |\lambda^i| = (r+1)(l+1)$.
If $\sum \#\lambda^i = 2(r+1)$,  then we have the equality $c_{\lambda^\bullet,(l)}^{1, (l^{r+1})} = c_{\lambda^\bullet}^{(l^{r+1},1^{r+1})},$
where the left-hand side is a
 degree 1 generalized quantum Littlewood-Richardson coefficient for $\QH^*\Gr_{r+1,r+1+l}$, as defined in \eqref{eq:quantum-product},
 and the right-hand side is 
a classical generalized Littlewood-Richardson coefficient, as defined in \eqref{eq:LR}. \end{lemma}
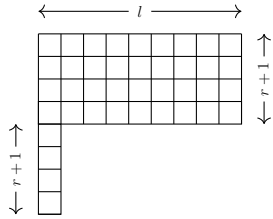
\begin{figure}[h!]
\centering
\begin{tikzpicture}[scale = .3]
\node[scale = 0.5] at (4.5, -1) {$l$};
\draw[->] (4, -1) -- (0, -1);
\draw[->] (5, -1) -- (9, -1);
\node[scale = .5, rotate = 90] at (10, -4) {$r+1$};
\draw[->] (10, -5) -- (10, -6);
\draw[->] (10, -3) -- (10, -2);
\node[scale = 0.5, rotate = 90] at (-1, -8) {$r+1$};
\draw[->] (-1, -7) -- (-1, -6);
\draw[->] (-1, -9) -- (-1, -10);
\draw (0, -5) rectangle (9, -6);
\draw (0, -6) rectangle (1, -10);
\draw (0, -2) -- (9, -2);
\draw (0, -3) -- (9, -3);
\draw (0, -4) -- (9, -4);
\draw (0, -2) -- (0, -6);
\draw (1, -2) -- (1, -6);
\draw (2, -2) -- (2, -6);
\draw (3, -2) -- (3, -6);
\draw (4, -2) -- (4, -6);
\draw (5, -2) -- (5, -6);
\draw (6, -2) -- (6, -6);
\draw (7, -2) -- (7, -6);
\draw (8, -2) -- (8, -6);
\draw (9, -2) -- (9, -6);
\draw (0, -7) -- (1, -7);
\draw (0, -8) -- (1, -8);
\draw (0, -9) -- (1, -9);
\draw (0, -10) -- (1, -10);
\end{tikzpicture}
\caption{The partition $\gamma\coloneqq(l^{r+1},1^{r+1})$ of Lemma~\ref{lem:quantum-classical}. Here $l = 9$ and $r = 3$.}
\label{gammafig}
\end{figure}

\begin{proof}
By Lemma~\ref{mainlemma}, we can calculate the quantum Littlewood-Richardson coefficient $c_{\lambda^\bullet,(l)}^{1, (l^{r+1})}$ by first computing the classical product of Schur polynomials and then removing rim-hooks to obtain $q\sigma_{(l^{r+1})}$. We are therefore interested in partitions $\beta$ of width at most $l$ that have a length $l+r+1$ rim-hook that produces the partition $ (l^{r+1})$. It follows from the definition of rim-hooks that the minimum length of such a partition $\beta$ is $2(r+1)+1$. However, if $\beta$ appears with non-zero coefficient in the classical product of the $\lambda^i$ and $(l)$, it also has length at most 
\[ \sum \# \lambda^i +1=2(r+1)+1.\]
It follows that the only possible $\beta$ that contribute have length precisely $2(r+1)+1$. There is exactly one such $\beta$:
 \[\gamma'=(\underbrace{l,\dots,l}_{r+2 \text{ times}},\underbrace{1,\dots,1}_{r+1 \text{ times}}),\]
 pictured in Figure~\ref{gammaprime}.
 
 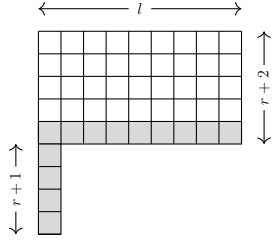
\begin{figure}
\centering
\begin{tikzpicture}[scale = 0.3]
\node[scale = 0.5] at (4.5, 0) {$l$};
\draw[->] (4, 0) -- (0, 0);
\draw[->] (5, 0) -- (9, 0);
\node[scale = 0.5, rotate = 90] at (10, -3.5) {$r+2$};
\draw[->] (10, -4.5) -- (10, -6);
\draw[->] (10, -2.5) -- (10, -1);
\node[scale = 0.5, rotate = 90] at (-1, -8) {$r+1$};
\draw[->] (-1, -7) -- (-1, -6);
\draw[->] (-1, -9) -- (-1, -10);
\draw[fill=gray!30] (0, -5) rectangle (9, -6);
\draw[fill=gray!30] (0, -6) rectangle (1, -10);
\draw (0, -1) -- (9, -1);
\draw (0, -2) -- (9, -2);
\draw (0, -3) -- (9, -3);
\draw (0, -4) -- (9, -4);
\draw (0, -1) -- (0, -6);
\draw (1, -1) -- (1, -6);
\draw (2, -1) -- (2, -6);
\draw (3, -1) -- (3, -6);
\draw (4, -1) -- (4, -6);
\draw (5, -1) -- (5, -6);
\draw (6, -1) -- (6, -6);
\draw (7, -1) -- (7, -6);
\draw (8, -1) -- (8, -6);
\draw (9, -1) -- (9, -6);
\draw (0, -7) -- (1, -7);
\draw (0, -8) -- (1, -8);
\draw (0, -9) -- (1, -9);
\draw (0, -10) -- (1, -10);
\end{tikzpicture}
\caption{The partition $\gamma'$ with its length $l + r+ 1$ rim-hook shaded in gray. In this example, $l = 9$ and $r = 3$.
The partition $\gamma$ of Figure~\ref{gammafig} is obtained by removing the top row of $\gamma'$.}
\label{gammaprime}
\end{figure}

 As the sign appearing in Lemma~\ref{mainlemma} is positive in this case, we precisely obtain the  generalized Littlewood-Richardson coefficient $c_{\lambda^\bullet,(l)}^{\gamma'} = c_{\lambda^\bullet}^{\gamma}$ where $\gamma = (l^{r+1},1^{r+1})$ is obtained from $\gamma'$ by removing a maximal row. 
\end{proof}

\subsection{Two-step flag varieties} \label{a:twostep}

We give some  constructions for and results about Schubert calculus on  two-step flag varieties, including computations using multiple bases for their cohomology rings.

Consider the flag variety $\Fl_{a,b;m}$ of nested subspaces $V\subset W$ in an $m$-dimensional vector space, where $\dim V=a$, \(\dim W=b\). The cohomology ring
$\H^*\Fl_{a,b;m}$ has a basis of Schubert classes indexed by permutations $w\in S_{m}$ such that $w(i)<w(i+1)$ for all $i\neq a,b$, i.e.\ the only  possible descents of $w$ are at positions $a,b$.
Following \cite{gukalashnikov}, we give an alternative indexing of these Schubert classes in terms of pairs of partitions $(\alpha,\beta)$ with $\alpha\subseteq (a^{b-a})$ and $\beta\subseteq (b^{m-b})$. We can also view $\alpha$ as a subset of $(a^{m-a})$. 

Let $w$ be a permutation indexing a Schubert class on $\Fl_{a,b;m}$. To find the pair of partitions $(\alpha,\beta)$ corresponding to $w$, we factor $w$ into two Grassmann permutations, $w=w_2 w_1.$
To define $w_1$ and $w_2$, we set $\rho$ to be the permutation that
\begin{itemize}
\item fixes $j>b$, i.e.\ $\rho(j)=j$ for all $j>b$, and
\item reorders $w(1),\dots,w(b)$ into increasing order, i.e.\ $w(\rho(1))<\cdots<w(\rho(b))$.
\end{itemize}
We set $w_2(i)\coloneqq w(\rho(i))$, and $w_1\coloneqq\rho^{-1}.$ Note that both $w_1$ and $w_2$ are Grassmann permutations, the first with a descent at $a$ and the second with a descent at $b$. As such, they define partitions $\alpha \subseteq (a^{b-a}) \subseteq (a^{m-a})$ and $\beta \subseteq (b^{m-b})$ respectively.

Given a pair of partitions $(\alpha,\beta)$, one can reverse this process to define a permutation $w_{\alpha,\beta}$. Note that the factorization of $w_{\alpha,\beta}$ above is precisely the factorization
\[w_{\alpha,\beta}=w_{\emptyset,\beta.}w_{\alpha,\emptyset}.\]

\begin{definition} The  \emph{inversion set} of a permutation $\rho$ is the set
$\{i<j \mid \rho(j)>\rho(i)\}.$
The number of inversions gives the codimension of the associated Schubert class.  
\end{definition}
\begin{remark}We will need the following observation in the proof of Proposition~\ref{pro:comparison}: the set  $\{i<j \leq b \mid w_{\mu,\nu}(j)>w_{\mu,\nu}(i)\}$ is the inversion set of $w_{\mu,\emptyset}$. 
\end{remark}

The pair $(\alpha,\beta)$ of partitions for $\Fl_{a,b;m}$ also corresponds to a product of Schubert classes in $\Gr_{a,b} \times \Gr_{b,m}$. 
In general, Schubert calculus on the flag variety behaves very differently than that on the product of Grassmannians, where it is governed by Littlewood-Richardson rules. However, in certain cases, these two products coincide. We prove the following equality:
\begin{proposition}\label{pro:comparison} Let $(\alpha^1,\beta^1),\dots,(\alpha^n,\beta^n)$ be $n$ pairs of partitions indexing Schubert classes of the flag variety $\Fl_{a,b;m}$. Suppose that
$\sum_{i=1}^n |\alpha^i|+|\beta^i|=a(b-a)+b(m-b)$.
If we have the inequality $\sum_{i=1}^n |\alpha^i| \leq a(b-a)$, then
\begin{equation}\label{eq:prop}\int_{\Fl_{a,b;m}} \prod_{i=1}^n \sigma_{(\alpha^i,\beta^i)}= \int_{\Gr_{a,b}} \prod_{i=1}^n \sigma_{\alpha^i} \int_{\Gr_{b,m}} \prod_{i=1}^n \sigma_{\beta^i}.
\end{equation}
In particular, if $\sum_{i=1}^n |\alpha^i| < a(b-a)$, this quantity is zero.
\end{proposition}
In the case when $\sum  |\alpha^i| = a(b-a)$, this recovers  \cite[Thm~1.1]{Richmond12} for $\Fl_{a,b;m}$; see also \cite{Richmond09} for related results.

In order to prove this proposition, we use both the Schubert basis as well as an alternative basis whose product rules are described in \cite{gukalashnikov}.  The alternative basis is indexed by the same pairs of partitions as the Schubert basis; the element corresponding to a pair $(\alpha,\beta)$ is $s_{\alpha,\beta}\coloneqq\sigma_{\alpha,\emptyset} \sigma_{\emptyset,\beta}$, which is a product of Schubert classes. This class can also be expressed as a product of Schur polynomials in the Chern roots of the tautological bundles of the flag variety.   We first state and prove an auxiliary result about the expansion of $s_{\alpha,\beta}$ in the Schubert basis.

\begin{lemma}\label{lem:comparison} Let $(\alpha,\beta)$ be a pair of partitions indexing a Schubert classes of the flag variety $\Fl_{a,b;m}$. Then
 \begin{align} \label{eq:comparison1} s_{\alpha,\beta} &=\sigma_{\alpha,\beta}+\sum_{(\mu,\nu),  |\mu| < |\alpha|} c_{\mu,\nu} \sigma_{\mu,\nu} \,  \text{ for some positive integers $c_{\mu,\nu}$, and }\\ \label{eq:comparison2} 
 \sigma_{\alpha,\beta} &=s_{\alpha,\beta}+\sum_{(\mu,\nu), |\mu| < |\alpha|} d_{\mu,\nu} s_{\mu,\nu} \,  \text{ for some  integers $d_{\mu,\nu}$}.
\end{align}
It follows that the top degree classes in both bases agree:  
\begin{equation}\label{eq:topdegree}
\sigma_{\alpha,\beta}= s_{\alpha,\beta} \, \text{ when } (\alpha,\beta)=(a^{b-a},b^{m-b}).
\end{equation}

\end{lemma}
\begin{proof} Note that as $s_{\alpha,\beta}=\sigma_{\alpha,\emptyset} \sigma_{\emptyset,\beta}$, the right hand side is simply the expansion in the Schubert basis of this product of two Schubert classes. Both terms in the product can be pulled back from Grassmannians. In particular, we can apply Proposition 2.3 of \cite{PurbhooSottile} to see that if $\sigma_{\mu,\nu}$ appears in the product of $\sigma_{\alpha,\emptyset} \sigma_{\emptyset,\beta}$ with non-zero coefficient, then the associated permutation $w_{\mu,\nu}$ satisfies the following condition: for all $i<j \leq b$, if $w_{\alpha,\emptyset}(i)<w_{\alpha,\emptyset}(j)$ then $w_{\mu,\nu}(i)<w_{\mu,\nu}(j)$.  

Thus the set $\{i<j \leq b \mid w_{\mu,\nu}(i)>w_{\mu,\nu}(j)\}$, which is precisely the inversion set of $w_{\mu,\emptyset}$, is a subset of the inversion set of $w_{\alpha,\emptyset}$. In particular $|\mu| \leq |\alpha|$, or equivalently, $|\nu| \geq |\beta|.$ To complete the proof the lemma, we need to show that if $|\mu|=|\alpha|$, then $\mu=\alpha$, $\nu=\beta$, and that the coefficient in the product of $\sigma_{\alpha,\beta}$ is $1$. Note that 
$\{i<j \leq b \mid w_{\mu,\nu}(i)>w_{\mu,\nu}(j)\}$
is precisely the inversion set of $\alpha$, so if $|\mu|=|\alpha|$, then $w_{\mu,\emptyset}$ and $w_{\alpha,\emptyset}$ have the same inversion set. As the inversion set completely determines the permutation, it follows that $\alpha=\mu.$

To see that $\nu=\beta$, note that Proposition 2.3 of \cite{PurbhooSottile} also implies that if $a<j$, then $w_{\mu,\nu}(j) \leq w_{\emptyset,\beta}(j)$. Since for $b<j$, $w_{\mu,\nu}(j)=w_{\emptyset,\nu}(j)$, it follows that 
 \[w_{\emptyset,\nu}(j) \leq w_{\emptyset,\beta}(j).\]
 This implies that $\beta \subseteq \nu$. Since we have assumed that $|\nu|=|\beta|$, in fact $\beta=\nu$.  
 We have shown that 
 \[s_{\alpha,\beta}=c \sigma_{\alpha,\beta}+\sum_{(\mu,\nu), |\mu| < |\alpha|} c_{\mu,\nu} \sigma_{\mu,\nu}\]
 for some constant $c$. Proposition 2.5 of \cite{PurbhooSottile} implies that $c=1$. 
Repeatedly applying \eqref{eq:comparison1}  gives  \eqref{eq:comparison2}.
Note that the coefficients $d_{\mu,\nu}$ in \eqref{eq:comparison2}. may not be positive. The equality \eqref{eq:topdegree} follows from the two statements.
\end{proof}

\begin{proof}[Proof of Proposition~\ref{pro:comparison}]  Let  $(\alpha^1,\beta^1),\dots,(\alpha^n,\beta^n)$ be $n$ pairs of partitions indexing Schubert classes in $\Fl_{a,b;m}$.

We first state some facts from  \cite{gukalashnikov} about multiplying in the alternative basis. For $s_{\alpha^1,\beta^1}$ and $s_{\alpha^2,\beta^2}$  two basis elements, by the rim-hook removal rule of \cite{gukalashnikov}, the product is governed by Littlewood-Richardson rules and rim-hook removals:
\begin{equation*} s_{\alpha^1,\beta^1} s_{\alpha^2,\beta^2} = \sum_{|\mu|=|\alpha^1|+|\alpha^2|} c^{\mu}_{\alpha^1, 
\alpha^2} c^{\nu}_{\beta^1, \beta^2} \, s_{\mu,\nu} + \sum_{|\mu|=|\alpha^1|+|\alpha^2|}  a_{\mu,\nu}\, s_{\mu,\nu}
\end{equation*}
for some integers $a_{\mu,\nu}$, where both sums are over $(\mu,\nu)$ satisfying $|\mu|+|\nu| = |\alpha^1|+|\alpha^2| +|\beta^1|+|\beta^2|$. Proceeding inductively and using properties of generalized Littlewood-Richardson coefficients, one can show
\begin{equation} \label{eq:gukalashnikov}
\prod_{i=1}^n s_{\alpha^i,\beta^i} = \sum_{|\mu|=\sum |\alpha^i|}  \prod_{i=1}^n c^{\mu}_{\alpha^\bullet} c^{\nu}_{\beta^\bullet} \, s_{\mu,\nu} + \sum_{|\mu|<\sum |\alpha^i|}  \tilde{a}_{\mu,\nu}\, s_{\mu,\nu}
\end{equation}
for some integers $\tilde{a}_{\mu,\nu}$, where both sums are over   $(\mu,\nu)$ satisfying $|\mu|+|\nu| =\sum |\alpha^i|+ |\beta^i|$. Note that  no summands appear where $|\mu|>\sum  |\alpha^i|$, or equivalently where $|\nu|<\sum |\beta^i|$.

 By \eqref{eq:comparison1}, \eqref{eq:comparison2} and \eqref{eq:gukalashnikov}, we obtain
 \begin{equation}
 \label{eq:prodschub}\prod_{i=1}^n \sigma_{\alpha^i,\beta^i}  = \sum_{|\mu|=\sum |\alpha^i|}  c^{\mu}_{\alpha^\bullet } c^{\nu}_{\beta^\bullet} \, s_{\mu,\nu}+\sum_{|\mu|<\sum |\alpha^i|} \tilde{d}_{\mu,\nu}\, s_{\mu,\nu}
 \end{equation}
for some integers $ \tilde{d}_{\mu,\nu}$, where both sums are over   $(\mu,\nu)$ satisfying $|\mu|+|\nu| =\sum |\alpha^i|+ |\beta^i|$. 

When
$\sum  |\alpha^i|+|\beta^i|=a(b-a)+b(m-b)$, the intersection number in the statement of the proposition can be read from \eqref{eq:prodschub} as the coefficient of the top degree class, namely of $\sigma_{(a^{b-a},b^{m-b})}= s_{(a^{b-a},b^{m-b})}$ (see \eqref{eq:topdegree}). By assumption, $\sum |\alpha_i| \leq a(b-a)$, so this top degree class does not appear  in the second summand of \eqref{eq:prodschub}, and therefore the intersection number is equal to  $c^{(a^{b-a)}}_{\alpha^\bullet } c^{(b^{m-b})}_{\beta^\bullet}$, which is in turn equal to the desired product of intersection numbers. The final statement of the proposition follows from the fact that $c^{(a^{b-a)}}_{\alpha^\bullet }=0$ if $\sum |\alpha^i| < a(b-a)$.
\end{proof}

  \section{$F$-curves and their intersections with divisors}
In our analysis of GW-divisors and the first Chern classes of critical level CB bundles, we compare their intersections with a set of curves in $\M_{0,n}$, defined next.

\begin{definition}\label{FCurves}
An $F$-curve  on $\M_{g,n}$ is the numerical equivalence class of an irreducible component of a one-dimensional component of the boundary.  
\end{definition}

$F$-Curves  on $\M_{0,n}$ are 
parametrized by  partitions $[n]=\{1,\ldots,n\} = N_1 \cup N_2 \cup N_3 \cup N_4$ as follows. For $i = 1, \ldots, 4$, let $X_{N_i}=(\mathbb{P}^1, P_{N_i}^\bullet \cup \alpha'_i) \in \op{M}_{0,|N_i| +1}$, be four fixed points. Define a map $\M_{0,4} \to \M_{0,n}$ by sending a point $X=(C,\alpha_{\bullet})\in \M_{0,4}$ to the $n$-pointed curve obtained by gluing the curve $X_{N_i}$ to $X$ by attaching $\alpha_i$ to the point $\alpha'_i$ for each $i\in \{1,\ldots,4\}$.  The $F$-curve, denoted $F_{N_1 N_2 N_3 N_4}$,  is defined to be the numerical equivalence class of the image of this map.  
The $F$-curves span $H_2(\M_{0,n}, \mathbb{Q})$, and are conjectured to span the extremal rays of the cone of curves on $\ovop{M}_{g,n}$.  This is known for $g=0$ and $n\le 7$, and for $n=0$ and $g\le 24$  \cite{gkm}, \cite{KeelMcKernan}, \cite{GibneyCompositio}.

Each component of the boundary is the surjective image of a morphism from a product of moduli spaces. To 
compute the intersection of the divisor with an $F$-curve, one pulls back the first Chern classes along these clutching morphisms: For $N_1 \subset \{1,\ldots, n\}$ a nonempty set, let $X_{N_1}=(\mathbb{P}^1,P^{\bullet}_{N_1}\cup \alpha_{N_1})\in \M_{0,|N_1|+1}$ be a smooth $|N_1|+1$-pointed rational curve, and define a morphism
\begin{equation}\label{Clutch}
\M_{0,|N_1^C|+1} \overset{F_{N_1}}{\twoheadrightarrow} \Delta_{N_1} \hookrightarrow \M_{0,n},
\end{equation}
attaching $X_{N_1}$ to a point $(C,P^{\bullet}_{N^C_1}\cup \alpha'_{N_1})\in \M_{0,|N^C_1|+1}$ by gluing $\alpha_{N_1}$ to $\alpha'_{N_1}$.
\begin{lemma}\label{CLELemma} Let $[n]= N_1 \cup N_2 \cup N_3 \cup N_4$  define an $F$-curve $F_{N_1,N_2,N_3,N_4}$. Then
\begin{multline}\label{CLEb}  c_1(\mathbb{V}(\sL_{r+1},\lambda^{\bullet}, l)) \cdot F_{N_1,N_2,N_3,N_4} \\
 = \sum_{\nu^{\bullet}} c_1(\mathbb{V}(\sL_{r+1},\nu^{\bullet}, l)) \ 
  \prod_{1\le j \le 4}{\rm{Rk}}(\mathbb{V}(\sL_{r+1}, \lambda(N_j)^{\bullet}\cup (\nu^j)^*, l)),\end{multline}
where one sums over  4-tuples $\nu^{\bullet}=\{\nu^j\}_{j=1}^4$ of  partitions of $\Gr_{r,r+l}$,  and   $(\nu^j)^*$ is defined to be the complement of $\nu^j$ in the $(r+1)\times \nu^j_1$ rectangle.
Each summand  is zero unless 
\begin{equation}\label{Constraint}
|\nu^i|=\sum_{j\in N_i}|\lambda^j|, \ \mbox{ for all } \ i \in \{1,2,3,4\}.
\end{equation}
\end{lemma}

The formula \eqref{CLEb} is well-known (see \cite{fakhr}, \cite{BG}). For completeness, we provide a proof, which also establishes
\eqref{Constraint}.
 This uses factorization, and the observation, known for some time, that boundary restrictions of  bundles at the critical level remain at or above the critical level.  

\begin{proof}To compute the intersection of $c_1(\mathbb{V}(\mathfrak{g},\lambda^{\bullet},l))$ with $F_{N_1,N_2,N_3,N_4}$, one pulls back the divisor along a composition of clutching maps as depicted in \eqref{Clutch}.
First, pulling back along $F_{N_1}$, we apply the factorization theorem \cite{tuy}, to obtain
\begin{multline}\label{PullBack}
F_{N_1}^*c_1(\mathbb{V}(\sL_{r+1}, \lambda^\bullet, l))\\
=\oplus_{\nu^1}
c_1(\mathbb{V}(\sL_{r+1}, \nu^1\cup \lambda(N_1^C)^{\bullet}, l )) {\rm{Rk}}(\mathbb{V}(\sL_{r+1}, (\nu^1)^{*} \cup \lambda(N_1)^{\bullet}, l)),
\end{multline}
where here we sum over partitions $\nu^1$ in an an $r\times l$ rectangle.  
In particular, to have bundles with nontrivial ranks and Chern classes, $r+1$ divides the total area of the partitions defining the modules for each bundle.  In other words:
\begin{equation}
\sum_{j\in N_1^c}|\lambda^j|+|\nu^1|=(r+1)(l+s_{1}), \ \mbox{ and } \sum_{j\in N_1}|\lambda^j|+|(\nu^1)^*|=(r+1)(l+s'_{1}).
\end{equation}
From  $\sum_{i\in [n]}|\lambda^i| =(r+1)(l+1)$ and $ |\nu^1| + |(\nu^1)^*| = (r+1)\nu^1_1$, we obtain
\[ |\nu^1| =  \sum_{i\in N_1}|\lambda^i| +(r+1)(s_1-1)\]
and $\nu^1_1=l+s_1+s'_1-1\leq l$ so that $s_{1}+s_{1}' \le 1$.

We wish to show that each summand of \eqref{PullBack} is zero unless $|\nu^1| = \sum_{i\in N_1}|\lambda^i|$. If $|\nu^1| < \sum_{i\in N_1}|\lambda^i|$,   then $s_1<1$, and the first Chern class component of the summand is $0$.
Note that if $|\nu^1|= \sum_{i\in N_1}|\lambda^i|$, then 
 $s_1=1$, and the first Chern class is at the critical level.
 
It remains to consider the summand in  \eqref{PullBack} when $|\nu^1| > \sum_{i\in N_1}|\lambda^i|$ and $s_1>1$. We will show that the rank component of the summand is zero:
 \begin{equation}\label{R1}R_1={\rm{Rk}}(\mathbb{V}(\sL_{r+1}, (\nu^1)^{*} \cup \lambda(N_1)^{\bullet}, l))=0.
 \end{equation}
By Witten's Dictionary, we may use a classical cohomology computation for $R_1$ since  \[|(\nu^1)^{*}|+\sum_{i\in N_1}|\lambda^i|=(r+1)(l+1-s_1)<(r+1)(l).\]
With $\lambda(N_1)=\{\gamma^1,\ldots, \gamma^k\}$, then $R_1$ is the coefficient of $\sigma_{((l+1-s_1)^{r+1})}$  in the product
\[\sigma_{\gamma^1}\cdot\sigma_{\gamma^2}\cdot\dots\cdot\sigma_{\gamma^k}  \cdot\sigma_{(\nu^1)^*} \in {H}^*\Gr_{r+1,r+1+l+s'_1}.\]
In this case, $\nu^1_1=l+s_1+s_1'-1>l+s_1'$ and $(\nu^1)^*$ has width $\nu_1^1$, so its cohomology class and hence the intersection is zero.

Since $[n] = N_1 \cup N_2 \cup N_3 \cup N_4$ is a partition into four nonempty sets, $N_2\subset N_1^C$. For $X_{N_2}=(\mathbb{P}^1,P_{N_2}^{\bullet},\alpha_{N_2})$ a point in $\M_{0,|N_2|+1}$, we can define a clutching map
\begin{equation}\label{Clutch2}
\M_{0,(|(N_1\cup N_2)^C|+2} \overset{F_{N_1, N_2}}{\twoheadrightarrow} \Delta_{N_2}\cap \Delta_{N_1} \hookrightarrow \M_{0,n},
\end{equation}
attaching the two points $X_{N_1}=(\mathbb{P}^1,P_{N_1}^{\bullet},\alpha_{N_1})$, and $X_{N_2}=(\mathbb{P}^1,P_{N_2}^{\bullet},\alpha_{N_2})$ to the point \[(C,P^{\bullet}_{|(N_1\cup N_2)^C|}\cup \alpha'_{N_1}\cup \alpha'_{N_2}) \in \M_{0,|(N_1\cup N_2)^C|+2}\] by identifying $\alpha_{N_2}$ and $\alpha'_{N_2}$. Factorization  gives $F_{N_1,N_2}^*(c_1(\mathbb{V}(\sL_{r+1},  \lambda^{\bullet}, l ))$ as a sum of divisors $c_1(\mathbb{V}(\sL_{r+1}, \nu^1\cup \nu^2 \cup  \lambda((N_1\cup N_2)^C)^{\bullet}, l ))$ that are at or above the critical level, with coefficients
\[{\rm{Rk}}(\mathbb{V}(\sL_{r+1}, (\nu^1)^{*} \cup \lambda(N_1)^{\bullet}, l)){\rm{Rk}}(\mathbb{V}(\sL_{r+1}, (\nu^2)^{*} \cup \lambda(N_2)^{\bullet}, l)),\] 
parametrized by partitions $\nu^1$ and $\nu^2$.  If both $|\nu^1|= \sum_{i\in N_1}|\lambda^i|$ and $|\nu^2|= \sum_{i\in N_2}|\lambda^i|$, then 
 the Chern class is at the critical level.  For this one checks $F_{N_1,N_2}^*(c_1(\mathbb{V}(\sL_{r+1},  \lambda^{\bullet}, l ))$ is a composition of clutching maps, and makes an analogous argument. Iterating, since $F_{N_1,N_2,N_3,N_4}$ represents the numerical equivalence class of a one dimensional component of $\Delta_{N_4} \cap \Delta_{N_3} \cap \Delta_{N_2}\cap \Delta_{N_1}$,
there is a clutching  map 
\begin{equation}\label{Clutch4}
 \M_{0,4}\overset{F_{N_{\bullet}}}{\twoheadrightarrow} F_{N_1,N_2,N_3,N_4}  \hookrightarrow \M_{0,n},
\end{equation}
attaching four fixed points $X_{N_i}=(\mathbb{P}^1,P_{N_i}^{\bullet}\cup \alpha^{N_i})\in \M_{0,|N_i|+1}$, to an arbitrary point $(C,Q^{\bullet})\in \M_{0,4}$ by identifying $\alpha^{N_i}$ and $Q^i$.
By Factorization,  $F_{N_{\bullet}}^*(c_1(\mathbb{V}(\sL_{r+1},  \lambda^{\bullet}, l ))$ is a sum of divisors $c_1(\mathbb{V}(\sL_{r+1}, \nu^1\cup \nu^2 \cup  \nu^3 \cup \nu^4, l ))$ that are at or above the critical level, with coefficients
\[\Pi_{j=1}^4{\rm{Rk}}(\mathbb{V}(\sL_{r+1}, (\nu^j)^{*} \cup \lambda(N_j)^{\bullet}, l)).\]
If $|\nu^j|= \sum_{i\in N_j}|\lambda^i|$, then 
 the Chern class is at the critical level.
\end{proof}

\begin{remark}\label{RankZero} It is well known that if $\sum_{i=1}^n|\lambda^i|$ is not divisible by $r+1$, then the rank of the bundle $\mathbb{V}(\sL_{r+1}, \lambda_{\bullet}, l)$ is zero.  This follows by induction on $n$ using the factorization theorem with base cases $n\in \{1,2,3\}$.
For  $n\in \{1,2\}$, the assertion is given by the fusion rules \cite[Cor~4.4]{BeaVer}.  For $n=3$, there are different ways to obtain the result.  For instance, one can also use the fusion rules, as is done in \cite[Proposition~3.4.]{ags} to show the claim for $\sL_2$ (see \cite[\S 5]{BeaVer}), although there isn't a closed form for these and one has to work them out for each $r$.  Alternatively, one may use \cite[Proposition~3.23]{tuy}, in which it is shown that there is a surjection from the 
 constant bundle $\mathbb{A}(\sL_{r+1}, \{\lambda_1,\lambda_2, \lambda_3\}, l)$, determined by the $\sL_{r+1}$-modules given by the partitions $\lambda_i$ onto
$\mathbb{V}(\sL_{r+1}, \{\lambda_1,\lambda_2, \lambda_3\}, l)$. Since defined on $\overline{\operatorname{M}}_{0,3}$, which is isomorphic to a point,  these are vector spaces. The vector space $\mathbb{A}(\sL_{r+1}, \{\lambda_1,\lambda_2, \lambda_3\}, l)$ is isomorphic to the tensor product of the highest weight $\sL_{r+1}$-modules determined by the partitions $\lambda_i$ \cite{ZhuModular}, which can be shown to be trivial unless $r+1$ divides $\sum_{i=1}^3|\lambda^i|$. 
  \end{remark}

\subsection*{Acknowledgements} This project began in July 2020 at the (virtual) ICERM Women in Algebraic Geometry workshop,  and we thank Melody Chan, Antonella Grassi, Rohini Ramadas, Julie Rana, and Isabel Vogt for organizing that workshop.  We  thank Anders Buch and Edward Richmond for helpful conversations.  We also thank Prakash Belkale for comments on the manuscript and note that many observations about critical level bundles originate from discussions with him.

\bibliography{Biblio}

\end{document}